\documentclass[a4paper,reqno]{amsart}

\usepackage[english]{babel}

\usepackage[T1]{fontenc}
\usepackage{lmodern}

\usepackage[tbtags]{mathtools}

\usepackage{amsfonts}
\usepackage{amssymb}
\usepackage{mathrsfs}

\usepackage[backend=bibtex,style=alphabetic,giveninits=true,maxnames=99]{biblatex}
\usepackage{csquotes}

\usepackage{enumitem}

\usepackage{tikz}
\usepackage{tikz-cd}
\usetikzlibrary{calc}
\usetikzlibrary{arrows}

\usepackage{xurl}
\usepackage[linktocpage]{hyperref}
\usepackage{amsthm}
\usepackage[capitalize]{cleveref}

\hypersetup{breaklinks=true} 

\NewBibliographyString{toappear}
\DefineBibliographyStrings{english}{%
  toappear = {to appear},
}
\renewbibmacro*{in:}{%
  \iffieldundef{pubstate}
    {}
    {\printfield{pubstate}%
     \setunit{\addspace}%
     \clearfield{pubstate}\bibstring{in}\addspace}%
  \ifentrytype{article}{}{\bibstring{in}\printunit{\intitlepunct}}}

\renewbibmacro*{volume+number+eid}{\printfield{volume}\setunit*{\addcomma}\printfield{number}\setunit{\bibeidpunct}\printfield{eid}}
\DeclareFieldFormat[article]{number}{ no. #1} 

\DeclareFieldFormat{eprint:arXiv}{%
  arXiv\addcolon
  \ifhyperref
    {\href{https://arxiv.org/abs/#1}{%
       \nolinkurl{#1}%
       \iffieldundef{eprintclass}
         {}
         {\addspace\texttt{\mkbibbrackets{\thefield{eprintclass}}}}}}
    {\nolinkurl{#1}%
     \iffieldundef{eprintclass}
       {}
       {\addspace\texttt{\mkbibbrackets{\thefield{eprintclass}}}}}}

\newtheorem{thm}{Theorem}
\newtheorem{introthm}{Theorem}
\newtheorem{lemma}[thm]{Lemma}
\newtheorem{prop}[thm]{Proposition}
\newtheorem{cor}[thm]{Corollary}

\theoremstyle{definition}
\newtheorem{dfn}[thm]{Definition}
\newtheorem{ex}[thm]{Example}

\theoremstyle{remark}
\newtheorem{rem}[thm]{Remark}

\numberwithin{thm}{section}						

\crefname{thm}{Theorem}{Theorems}
\Crefname{thm}{Theorem}{Theorems}
\crefname{introthm}{Theorem}{Theorems}
\Crefname{introthm}{Theorem}{Theorems}
\crefname{lemma}{Lemma}{Lemmas}
\Crefname{lemma}{Lemma}{Lemmas}
\crefname{prop}{Proposition}{Propositions}
\Crefname{prop}{Proposition}{Propositions}
\crefname{cor}{Corollary}{Corollaries}
\Crefname{cor}{Corollary}{Corollaries}
\crefname{dfn}{Definition}{Definitions}
\Crefname{dfn}{Definition}{Definitions}
\crefname{ex}{Example}{Examples}
\Crefname{ex}{Example}{Examples}
\crefname{rem}{Remark}{Remarks}
\Crefname{rem}{Remark}{Remarks}
\crefname{equation}{}{}
\Crefname{equation}{}{}
\crefname{enumi}{}{}
\Crefname{enumi}{}{}
\crefname{tfaei}{}{}
\Crefname{tfaei}{}{}

\setlist[itemize]{itemsep=0mm}
\setlist[enumerate]{itemsep=0mm}
\setlist[enumerate,1]{label=\arabic*)}	

\newlist{tfae}{enumerate}{1}
\setlist[tfae]{itemsep=0mm}
\setlist[tfae]{label=\Roman*)}

\newcommand{\ints}{\mathbb{Z}}

\newcommand{\groundfield}{\Bbbk}

\newcommand{\ol}[1]{\overline{#1}}

\renewcommand{\:}{\colon}

\newcommand{\imp}{\implies}		
\newcommand{\pmi}{\impliedby}	

\newcommand{\hookto}{\hookrightarrow}							
\newcommand{\tto}{\rightarrow\mathrel{\mkern-14mu}\rightarrow}	

\newcommand{\longto}[1]{\xrightarrow{#1}}						

\newenvironment{smallpmatrix}{\left(\begin{smallmatrix}}{\end{smallmatrix}\right)}

\newcommand{\cat}[1]{\mathscr{#1}}				

\DeclareMathOperator{\id}{id}					

\DeclareMathOperator{\cone}{cone}				

\DeclareMathOperator{\tria}{tria}				
\DeclareMathOperator{\thick}{thick}				
\DeclareMathOperator{\extclos}{extclos}			
\DeclareMathOperator{\Kar}{Kar}					

\newcommand{\rMod}{\mathbf{Mod}\textrm{-}}		
\newcommand{\rmodfd}{\mathbf{mod}_{\mathrm{fd}}\textrm{-}}		
\newcommand{\rprojfg}{\mathbf{proj}_{\mathrm{fg}}\textrm{-}}	
\newcommand{\rinjfg}{\mathbf{inj}_{\mathrm{fg}}\textrm{-}}		

\newcommand{\Proj}{\mathbf{Proj}}			
\newcommand{\DProj}{\mathbf{DProj}}			

\newcommand{\dgrMod}{\mathbf{dgMod}\textrm{-}}	

\newcommand{\heart}{\heartsuit}							
\newcommand{\coheart}{\raisebox{\depth}{\scalebox{1}[-1]{$\heartsuit$}}}	

\newcommand{\D}{\mathbf{D}}						
\newcommand{\Db}{\mathbf{D}^\mathrm{b}}			
\newcommand{\Dfd}{\mathbf{D}_{\mathrm{fd}}}		

\newcommand{\K}{\mathbf{K}}						
\newcommand{\Kb}{\mathbf{K}^\mathrm{b}}			

\newcommand{\perf}{\mathbf{perf}}				

\newcommand{\KbCAT}{\mathbb{K}^b}				
\newcommand{\DbCAT}{\mathbb{D}^b}				

\newcommand{\Hom}{\mathrm{Hom}}					
\newcommand{\End}{\mathrm{End}}					
\newcommand{\Ext}{\mathrm{Ext}}					

\newcommand{\RHom}{\mathrm{RHom}}				
\newcommand{\REnd}{\mathrm{REnd}}				

\newcommand{\Lotimes}{\otimes^\mathrm{L}}		

\newcommand{\simplem}{\mathcal{L}}				
\newcommand{\silting}{\mathcal{P}}				

\newcommand{\comp}{\mathrm{c}}					
\newcommand{\dg}{\mathrm{dg}}					
\newcommand{\coop}{\mathrm{coop}}				

\title[Derived projectives and simple-minded\slash silting Koszul duality]{Derived projective covers and Koszul duality of simple-minded and silting collections}

\author{Lukas Bonfert}
\address{Max Planck Institute for Mathematics, Bonn, Germany}
\email{bonfert@mpim-bonn.mpg.de}

\date{\today}



\begin{document}
	\begin{abstract}
		We introduce derived projective covers and explain how they are related to the notion of enough derived projectives.
		This provides an if-and-only-if criterion for when derived projective covers form a silting collection.
		We prove moreover a Koszul duality result for silting and simple-minded collections.
	\end{abstract}
	\maketitle
	\section{Introduction}
		In any triangulated category $\cat{C}$ there is a bijection between simple-minded collections in $\cat{C}$ and bounded \textbf{t}-structures with finite-length heart on $\cat{C}$, sending a \textbf{t}-structure $t$ to the set of simple objects in its heart $\heart_t$.
		Similarly, there is a bijection between (classical) silting collections in $\cat{C}$ and bounded weight structures (also known as co-\textbf{t}-structures) with Krull--Schmidt coheart, sending a weight structure $w$ to the set of indecomposable objects in its coheart $\coheart_w$.
		
		Weight structures and \textbf{t}-structures on (not necessarily the same) triangulated categories are closely related by the notion of \emph{orthogonality} introduced in \cite{BondarkoMotivicSpectralSequences}.
		For instance, for a finite-dimensional algebra $A$ \cite{KoenigYang} establishes a bijection between bounded weight structures on $\Kb(\rprojfg A)$ and bounded \textbf{t}-structures on $\Db(\rmodfd A)$ with finite-length heart, and this can be rephrased in terms of \textbf{w}-\textbf{t}-strict left orthogonality.
		Here we say that a weight structure $w$ is \textbf{w}-\textbf{t}-strictly left orthogonal to a \textbf{t}-structure $t$ if $(\cat{C}_{w\leq 0})^\perp=\cat{D}^{t>0}$ and $\prescript{\perp}{}{(\cat{D}^{t>0})}=\cat{C}_{w\leq 0}$, and similarly for $\cat{C}_{w\geq 0}$ and $\cat{D}^{t<0}$.
		This is slightly stronger than the strict orthogonality considered in \cite{BondarkoWeightTAndBack}.
		In terms of the corresponding silting collection $\silting$ and simple-minded collection $\simplem$, the bijection from \cite{KoenigYang} is characterized by the existence of a bijection $\phi\: \silting\to\simplem$ such that
		\begin{equation}
			\label{simplemsiltingrelintro}
			\Hom(P,L[m])\cong\begin{cases}
				\End(L)&\text{if $L=\phi(P)$, $m=0$},\\
				0&\text{otherwise}.
			\end{cases}
		\end{equation}
		This bijection sends $P\in\silting$ to the simple top of $t_{\geq 0}P\in\heart_t$.
		Recently, \cite{Fushimi} proved such bijections in the setup of \emph{ST pairs} from \cite{AdachiMizunoYang}.
		This provides a common generalization of the results from \cite{KoenigYang} as well as the analogous results for non-positive dg algebras with finite-dimensional total cohomology from \cite{BruestleYang} and homologically smooth non-positive dg algebras from \cite{KellerNicolasCluster}.
		For positive dg algebras, the results in \cite{KellerNicolas} provide further examples of ST pairs.
		Here we call a dg algebra $A$ \emph{(cohomologically) non-positive} if $H^n(A)=0$ for $n>0$, and \emph{(cohomologically) positive} if $H^n(A)=0$ for $n<0$ and $H^0(A)$ is semisimple.
		
		Our goal is to study the orthogonality relation in more detail.
		In particular, we consider the following three aspects, which are mostly independent of each other:
		\begin{enumerate}
			\item Characterization of orthogonality in terms of simple-minded and silting collections via derived projective objects (\cref{derivedprojectives}).
			\item Koszul duality between simple-minded and silting collections (\cref{koszulduality}).
			\item Naturality of orthogonality (\cref{naturality}).
		\end{enumerate}
		The paper is structured as follows.
		In \cref{definitions} we recall the required definitions.
		We consider various strictness levels of orthogonality, thereby refining the picture from \cite{BondarkoWeightTAndBack}.
		However, by \cref{orthogonalitylemma,adjacencylemma} (which mostly follow from arguments in \cite{BondarkoWeightTAndBack}) in many cases at least some of these coincide.
		
		Our first result relates silting collections to derived projective objects (also known as Ext-projective objects) which are an analog of projective objects in the triangulated setting.
		Analogously to the setting of abelian categories, \cite{GenoveseLowenvdBergh} introduced the term \emph{enough derived projectives}.
		In \cref{enoughderivedproj} we show that this definition agrees with the notion of \emph{enough Ext-projectives} from \cite{CoelhoSimoesPauksztelloPloog}.
		In \cite[Thm.~2.4]{CoelhoSimoesPauksztelloPloog} it is shown that this provides a criterion for the existence of a left adjacent weight structure, see \cref{derprojadjacency} for a slightly refined version.
		
		For a \textbf{t}-structure $t_\silting=(\silting^{\perp_{>0}},\silting^{\perp_{<0}})$ obtained from a silting collection $\silting$ in the sense of \cite{PsaroudakisVitoria}, an easy but important observation shows that $\silting$ consist of derived projective objects with respect to $t_\silting$.
		As is evident from \cref{simplemsiltingrelintro}, the relation of silting collections to simple-minded collections is somewhat similar to the relation of indecomposable projective objects to simple objects in finite-length abelian categories.
		To formalize this, in \cref{derprojcoverdef} we introduce \emph{derived projective covers}, and we show the following result:
		\begin{introthm}[\cref{siltingderprojconditions}]
			\label{siltingderprojconditionsintro}
			Let $t$ be a non-degenerate \textbf{t}-structure on $\cat{D}$ with finite-length heart.
			Let $\simplem$ be a full set of isomorphism representatives of the simple objects in $\heart_t$ and $\silting$ a full set of isomorphism representatives of the indecomposable derived projectives.
			Then the following are equivalent:
			\begin{tfae}
				\item $t$ is silting (and $\silting$ is the silting collection).
				\item There is a bijection $\phi\: \silting\to\simplem$ satisfying \cref{simplemsiltingrelintro}.
				\item Every $L\in\simplem$ admits a derived projective cover (and $\silting$ is the set of these derived projective covers).
				\item $\cat{D}$ has enough derived projectives with respect to $t$.
			\end{tfae}
		\end{introthm}
		This result is somewhat analogous to \cite[Thm.~2.4]{CoelhoSimoesPauksztelloPloog}.
		As an application, in \cref{simplemindedsiltingorthogonal} we show that for a \textbf{t}-structure obtained from a simple-minded collection $\simplem$ and a weight structure obtained from a silting collection $\silting$, the relations \cref{simplemsiltingrelintro} characterize orthogonality, and moreover that these relations are equivalent to $\silting$ consisting of the derived projective covers of $\simplem$.
		This is not very surprising, as it is similar to the results contained in \cite{KoenigYang}.
		
		\Cref{siltingderprojconditionsintro} also allows us to study the bijections between weight structures and \textbf{t}-structures in more detail.
		For this, in \cref{wtpairdef} we introduce equivalent axioms for the ST pairs from \cite{AdachiMizunoYang}, and rename them to \emph{WT pairs} as this definition only uses weight structures and \textbf{t}-structures, but no silting collections.
		In \cref{weighttbijections} we show that at the level of weight structures and \textbf{t}-structures the bijection from \cite{Fushimi} is given by \textbf{w}-\textbf{t}-strict orthogonality.
		
		Our second main result is related to the apparent duality in the definition of weight structures and \textbf{t}-structures, and more precisely to a tentative Koszul duality of these notions.
		To motivate this, recall that the classical Koszul duality from \cite{BeilinsonGinzburgSoergel}, \cite{MazorchukOvsienkoStroppel} provides an equivalence of graded derived categories that swaps simples and indecomposable projectives.
		For a formulation in terms of \textbf{t}-structures and weight structures see \cite[\S 2.4]{EberhardtStroppel}.
		
		In our setup, the analogs of simple objects and indecomposable projective objects are simple-minded collections and silting collections, as evidenced by the standard examples and \cref{siltingderprojconditionsintro}.
		As a formalization of the tentative Koszul duality, Bernhard Keller suggested the following result, which uses the dg Koszul duality from \cite{KellerDerivingDGCats}:
		\begin{introthm}[\cref{koszuldualitygeneral}]
			\label{koszuldualitygeneralintro}
			Let $\cat{T}=H^0(\widetilde{\cat{T}})$ be a compactly generated dg-enhanced triangulated category, and let $\silting$ be a compact silting collection in $\cat{T}$ such that $\End_\cat{T}(\bigoplus_{P\in\silting} P)$ is finite-dimensional.
			Let $\simplem$ be the set of simple objects in the heart of the silting \textbf{t}-structure associated with $\silting$, and suppose that $\End_\cat{T}(L)$ is $1$-dimensional for each $L\in\simplem$. 
			Note that $\simplem$ is a simple-minded collection in $\cat{D}=\tria_\cat{T}(\simplem)$.
			\begin{enumerate}
				\item The dg algebra $\End_{\widetilde{\cat{T}}}(\bigoplus_{L\in\simplem}L)$ is the dg Koszul dual of $\End_{\widetilde{\cat{T}}}(\bigoplus_{P\in\silting}P)$.
				\item If $H^n(\End_{\widetilde{\cat{T}}}(\bigoplus_{P\in\silting}P))$ is finite-dimensional for all $n\in\ints$, then $\End_{\widetilde{\cat{T}}}(\bigoplus_{P\in\silting}P)$ is the dg Koszul dual of $\End_{\widetilde{\cat{T}}}(\bigoplus_{L\in\simplem}L)$.
			\end{enumerate}
		\end{introthm}
		This is inspired by \cite{BruestleYang}, see also the revised version \cite{BruestleYangUpdate}.
		In the case of finite-dimensional algebras or non-positive dg algebras with finite-dimensional total cohomology, the second part of \cref{koszuldualitygeneralintro} can also be shown by a construction from \cite{Zhang}, which was used there to construct a silting collection corresponding to a simple-minded collection.
		
		In the last section we show that \textbf{w}-\textbf{t}-strict orthogonality between weight structures and \textbf{t}-structures is natural with respect to weight exact functors and \textbf{t}-exact functors that are (in a certain sense) adjoint to each other.
		The setup of the main result \cref{orthogonalitynaturality} is somewhat technical, although the proof is straightforward and essentially the same as \cite[Prop.~4.4.5]{BondarkoWeightTAndBack}.
		In particular, it follows from this that the bijection between weight structures and \textbf{t}-structures from \cite{KoenigYang} is natural, see \cref{naturalityfdalg}.
		\subsection*{Acknowledgements.}
			This paper is part of my ongoing PhD project.
			I would like to thank Catharina Stroppel for many helpful discussions and feedback on various draft versions, and Bernhard Keller for sharing his ideas about Koszul duality between weight structures and \textbf{t}-structures, and further helpful comments.
			Further thanks goes to Lidia Angeleri H\"ugel and Jorge Vit\'oria for making me aware of \cite{GenoveseLowenvdBergh}, and for asking me the right questions about derived projectives.
			Furthermore, I thank the reviewers for their detailed comments.
			
			This project was supported by the Max Planck Institute for Mathematics (IMPRS Moduli Spaces) and the Hausdorff Center for Mathematics, which is funded by the Deutsche Forschungsgemeinschaft (DFG, German Research Foundation) under Germany’s Excellence Strategy  -- EXC-2047/1 -- 390685813.
	\section{Definitions}\label{definitions}
		We begin by recalling the definitions of \textbf{t}-structures, weight structures, simple-minded collections and silting collections.
		For silting collections, we also compare two slightly different definitions.
		Finally we recall the notion of orthogonality between weight structures and \textbf{t}-structures, and compare the various strictness levels in special cases.
		
		Unless explicitly mentioned, all categories will be linear over some (fixed) field $\groundfield$.
		For subcategories $\cat{A},\cat{B}\subseteq\cat{C}$ we write $\cat{A}\perp\cat{B}$ if $\Hom_\cat{C}(A,B)=0$ for all $A\in\cat{A}$, $B\in\cat{B}$.
		Moreover we write $\cat{A}^\perp=\{C\in\cat{C}\mid \Hom_\cat{C}(A,C)=0\;\forall A\in\cat{A}\}$ and $\prescript{\perp}{}{\cat{A}}=\{C\in\cat{C}\mid \Hom_\cat{C}(C,A)=0\;\forall A\in\cat{A}\}$.
		(Dg) modules over a (dg) algebra will be right modules, unless stated otherwise.
		\subsection{\texorpdfstring{\textbf{t}-structures}{t-structures}}
			The notion of \textbf{t}-structures on triangulated categories was introduced in \cite{BBD}.
			\begin{dfn}
				A \emph{\textbf{t}-structure} on a triangulated category $\cat{D}$ is a pair $t=(\cat{D}^{t\leq 0},\cat{D}^{t\geq 0})$ of strict full subcategories such that
				\begin{itemize}
					\item $\cat{D}^{t\leq 0}[1]\subseteq\cat{D}^{t\leq 0}$ and $\cat{D}^{t\geq 0}[-1]\subseteq\cat{D}^{t\geq 0}$,
					\item $\cat{D}^{t\leq 0}\perp\cat{D}^{t\geq 0}[-1]$,
					\item for all $X\in\cat{D}$ there is a triangle (called \emph{\textbf{t}-decomposition} of $X$)
						\begin{equation*}
							t_{\leq 0}X\to X\to t_{>0}X\to t_{\leq 0}X[1]
						\end{equation*}
						with $t_{\leq 0}X\in\cat{D}^{t\leq 0}$ and $t_{>0}X\in\cat{D}^{t\geq 0}[-1]$.
				\end{itemize}
				The full subcategory $\heart_t=\cat{D}^{t\leq 0}\cap\cat{D}^{t\geq 0}$ is called the \emph{heart} of $t$.
			\end{dfn}
			We also write $\cat{D}^{t>0}=\cat{D}^{t\geq 0}[-1]$ and $\cat{D}^{t<0}=\cat{D}^{t\leq 0}[1]$, and also $\cat{D}^{t\geq n}=\cat{D}^{t\geq 0}[-n]$ for $n\in\ints$ (and analogously $\cat{D}^{t\leq n}$).
			
			A \textbf{t}-structure $t$ is called \emph{non-degenerate} if $\bigcap_{n\in\ints} \cat{D}^{t\leq n}=\{0\}=\bigcap_{n\in\ints} \cat{D}^{t\geq n}$.
			It is \emph{bounded above} if $\cat{D}=\bigcup_{n\in\ints} \cat{D}^{t\leq n}$, \emph{bounded below} if $\cat{D}=\bigcup_{n\in\ints} \cat{D}^{t\geq n}$, and \emph{bounded} if $\cat{D}=\tria_\cat{D}(\heart_t)$.
			Here $\tria_\cat{D}(\heart_t)$ denotes the triangulated subcategory of $\cat{D}$ generated by $\heart_t$.
			Note that $t$ is bounded if and only if it is bounded above and bounded below.
			
			Recall that \textbf{t}-decompositions are unique up to isomorphism, and furthermore $t_{\geq 0}\: \cat{D}\to \cat{D}^{t\geq 0}$ and $t_{\leq 0}\: \cat{D}\to\cat{D}^{t\leq 0}$ define functors that are left (resp.~right) adjoint to the respective inclusions \cite[Prop.~1.3.3]{BBD}.
			Also recall that $\cat{D}^{t\leq 0}=\prescript{\perp}{}{(\cat{D}^{t>0})}$ and $\cat{D}^{t\geq 0}=(\cat{D}^{t<0})^\perp$.
		\subsection{Weight structures}
			Weight structures (also known as co-\textbf{t}-structures) were originally defined in \cite{Bondarko} and \cite{Pauksztello}.
			\begin{dfn}
				A \emph{weight structure} on a triangulated category $\cat{C}$ is a pair $w=(\cat{C}_{w\leq 0},\cat{C}_{w\geq 0})$ of Karoubi-closed full subcategories such that
				\begin{itemize}
					\item $\cat{C}_{w\leq 0}[1]\subseteq\cat{C}_{w\leq 0}$ and $\cat{C}_{w\geq 0}[-1]\subseteq\cat{C}_{w\geq 0}$,
					\item $\cat{C}_{w\geq 0}[-1]\perp\cat{C}_{w\leq 0}$,
					\item for all $X\in\cat{C}$ there is a triangle (called \emph{weight decomposition} of $X$)
						\begin{equation*}
							w_{>0}X\to X\to w_{\leq 0}X\to w_{>0}X[1]
						\end{equation*}
						with $w_{>0}X\in\cat{C}_{w\geq 0}[-1]$ and $w_{\leq 0}X\in\cat{C}_{w\leq 0}$.
				\end{itemize}
				The full subcategory $\coheart_w=\cat{C}_{w\leq 0}\cap\cat{C}_{w\geq 0}$ is called the \emph{coheart} of $w$.
			\end{dfn}
			As for \textbf{t}-structures we write $\cat{C}_{w>0}=\cat{C}_{w\geq 0}[-1]$, $\cat{C}_{w\geq n}=\cat{C}_{w\geq 0}[-n]$, and so on.
			
			A weight structure $w$ is called \emph{non-degenerate} if $\bigcap_{n\in\ints} \cat{C}_{w\leq n}=\{0\}=\bigcap_{n\in\ints} \cat{C}_{w\geq n}$.
			It is \emph{bounded above} if $\cat{C}=\bigcup_{n\in\ints} \cat{C}_{w\leq n}$, \emph{bounded below} if $\cat{C}=\bigcup_{n\in\ints} \cat{C}_{w\geq n}$, and \emph{bounded} if $\cat{C}=\thick_\cat{C}(\coheart_w)$.
			Here $\thick_\cat{C}(\coheart_w)$ denotes the thick subcategory of $\cat{C}$ generated by $\coheart_w$.
			Note that $w$ is bounded if and only if it is bounded above and bounded below.
			
			Analogously to the situation for \textbf{t}-structures we have $\cat{C}_{w\leq 0}=(\cat{C}_{w>0})^\perp$ and $\cat{C}_{w\geq 0}=\prescript{\perp}{}{(\cat{C}_{w<0})}$, see \cite[Prop.~1.3.3]{Bondarko}.
			In contrast to \textbf{t}-decompositions, by \cite[Rem.~1.2.2]{Bondarko} weight decompositions are usually not unique.
			In particular $w_{\leq 0}$ and $w_{>0}$ do not define functors.
		\subsection{Simple-minded collections}
			The following definition is from \cite{AlNofayee} and \cite{KoenigYang}, and the axioms already appeared in \cite{RickardEquivalencesDerivedCategories}.
			\begin{dfn}
				A \emph{simple-minded collection} in a triangulated category $\cat{D}$ is a (not necessarily finite) set $\simplem$ of objects of $\cat{D}$ such that
				\begin{itemize}
					\item $\Hom_\cat{D}(L,L'[m])=0$ for all $L,L'\in\simplem$, $m<0$,
					\item $\Hom_\cat{D}(L,L')=0$ for $L,L'\in\simplem$, $L\neq L'$,
					\item $\End_\cat{D}(L)$ is a division algebra for all $L\in\simplem$,
					\item $\tria_\cat{D}(\simplem)=\cat{D}$.
				\end{itemize}
				A simple-minded collection is \emph{finite} if it consists of finitely many objects.
			\end{dfn}
			\begin{rem}
				Note that in contrast to most of the existing literature we do not assume simple-minded collections to be finite, see also \cite{Schnuerer} where infinite simple-minded collections are also studied.
				However, if a triangulated category $\cat{D}$ admits a finite simple-minded collection, then automatically any simple-minded collection in $\cat{D}$ is finite, since it follows from \cref{ttosimplem} that simple-minded collections form bases of the Grothendieck group of $\cat{D}$.
			\end{rem}
			An abelian category is \emph{finite-length} (or a \emph{length category}) if all of its objects have finite length.
			The definition of simple-minded collections is based on properties of the simple objects in the heart of a bounded \textbf{t}-structure with finite-length heart, and in fact specifying a simple-minded collection is equivalent to specifying such a \textbf{t}-structure.
			This is already mentioned in \cite[Rem.~1.3.14]{BBD}, and explicitly spelled out in \cite{AlNofayee}.
			\begin{prop}\leavevmode
				\label{ttosimplem}
				\begin{enumerate}
					\item Let $\simplem$ be a simple-minded collection in $\cat{D}$.
						Then $t=(\cat{D}^{t\leq 0},\cat{D}^{t\geq 0})$, with
						\begin{align*}
							\cat{D}^{t\leq 0}&=\extclos_\cat{D}\{L[m]\mid L\in\simplem, m\geq 0\},\\
							\cat{D}^{t\geq 0}&=\extclos_\cat{D}\{L[m]\mid L\in\simplem, m\leq 0\},
						\end{align*}
						is a bounded \textbf{t}-structure with finite-length heart, and $\simplem$ is a full set of isomorphism representatives of the simple objects in $\heart_t$.
					\item Let $t$ be a bounded \textbf{t}-structure on $\cat{D}$ such that $\heart_t$ is finite-length, and let $\simplem$ be a full set of isomorphism representatives of the simple objects in $\heart_t$.
						Then $\simplem$ is a simple-minded collection in $\cat{D}$.
				\end{enumerate}
			\end{prop}
			\begin{proof}
				See \cite[Prop.~2 and Prop.~4]{AlNofayee}.
				Although the propositions there are formulated only for the bounded derived category of a self-injective algebra, the proofs work in a general triangulated category without modifications.
			\end{proof}
		\subsection{Silting collections}
			For a collection of objects $\mathcal{X}$ in a triangulated category $\cat{D}$ we write $\mathcal{X}^{\perp_{>0}}=\{D\in\cat{D}\mid\Hom_\cat{D}(X,D[m])=0\;\forall m>0, X\in\mathcal{X}\}$, and analogously define $\mathcal{X}^{\perp_{<0}}$, $\mathcal{X}^{\perp_{\geq 0}}$, etc.
			Moreover, $\Kar_\cat{D}(\mathcal{X})$ denotes the full subcategory whose objects are the direct summands of finite coproducts of objects in $\mathcal{X}$.
			
			The following definition is based on \cite[Def.~4.1]{PsaroudakisVitoria}.
			There are other definitions of silting, see \cite[Ex.~4.2]{PsaroudakisVitoria} for an overview and comparison of different definitions.
			\begin{dfn}
				\label{siltingdef}
				A \emph{silting collection} in a triangulated category $\cat{D}$ is a (not necessarily finite) set $\silting$ of objects of $\cat{D}$ such that
				\begin{itemize}
					\item $\Kar_\cat{D}(\silting)$ is Krull--Schmidt,
					\item objects in $\silting$ are indecomposable and pairwise non-isomorphic,
					\item $t_\silting=(\silting^{\perp_{>0}},\silting^{\perp_{<0}})$ is a \textbf{t}-structure on $\cat{D}$, called the \emph{silting \textbf{t}-structure associated with $\silting$}.
				\end{itemize}
				We say that $\silting$ is \emph{finite} if it consists of finitely many objects.
				A silting collection consisting of compact objects (in a triangulated category with small coproducts) is called \emph{compact}.
			\end{dfn}
			\begin{rem}
				\label{siltinghomnegativity}
				In \cite[Def.~4.1]{PsaroudakisVitoria} it is moreover required that $\Hom_\cat{D}(P,P'[m])=0$ for all $P,P'\in\silting$ and $m>0$.
				However, as mentioned in \cite[Prop.~2.5]{AHLSV}, this assumption is automatic:
				for $P\in\silting$, take the \textbf{t}-decomposition $t_{\leq 0}P\to P\to t_{>0}P\to t_{\leq 0}P[1]$.
				Then $t_{>0}P\in\silting^{\perp_{\leq 0}}$, so $t_{\leq 0}P[1]\cong P[1]\oplus t_{>0}P$ and thus $P\cong t_{\leq 0}P\in\silting^{\perp_{>0}}$.
			\end{rem}
			In the literature usually silting objects (rather than silting collections) are used, see e.g.~\cite{KoenigYang,PsaroudakisVitoria}.
			However these provide exactly the same data, at least in the finite case:
			given a silting object $P$, (isomorphism representatives of) its indecomposable summands form a silting collection.
			Conversely, if $\silting$ is a silting collection, then $\coprod_{P\in\silting} P$ is a silting object (assuming the coproduct exists).
			We prefer to use silting collections rather than silting objects since we are mostly interested in the indecomposable summands.
			\begin{rem}
				It is important to specify the ambient triangulated category $\cat{D}$ for a silting collection $\silting$.
				Note that if $\silting$ is a silting collection in $\cat{D}$, then not necessarily $\cat{D}=\thick_\cat{D}(\silting)$.
				In particular, $\silting$ is in general not a silting collection in $\thick_\cat{D}(\silting)$, since the associated silting \textbf{t}-structure need not restrict to a \textbf{t}-structure on $\thick_\cat{D}(\silting)$.
			\end{rem}
			Using the terminology from \cite[\href{https://stacks.math.columbia.edu/tag/09SJ}{Tag~09SJ}]{StacksProject}, we say that a set of objects $\mathcal{X}$ \emph{weakly generates} $\cat{D}$ if $\Hom_\cat{D}(X,Y[n])=0$ for all $n\in\ints$ and $X\in\mathcal{X}$ implies $Y=0$.
			The following result is also stated in \cite[Prop.~4.3]{PsaroudakisVitoria}, however we were unable to verify their proof.
			\begin{lemma}
				\label{siltingtnondeg}
				A silting collection $\silting$ in $\cat{D}$ weakly generates $\cat{D}$.
				In particular, the associated silting \textbf{t}-structure $t_\silting$ is non-degenerate.
			\end{lemma}
			\begin{proof}
				Let $X\in\cat{D}$ such that $\Hom_\cat{D}(P,X[n])=0$ for all $P\in\silting$ and $n\in\ints$.
				Then in particular $X\in\silting^{\perp_{>0}}$ and $X\in\silting^{\perp_{\leq 0}}$, and thus both $X\to X\to 0\to X[1]$ and $0\to X\to X\to 0$ are \textbf{t}-decompositions of $X$ with respect to $t_\silting$.
				But since \textbf{t}-decompositions are unique it follows that $X=0$.
				
				That $t_\silting$ is non-degenerate is equivalent to $\silting$ weakly generating $\cat{D}$ since $\cat{D}^{t_\silting\leq n}=\silting^{\perp_{>n}}$ for all $n\in\ints$, and analogously for the positive part.
			\end{proof}
			For us the main class of examples will be silting collections according to the following ``classical'' definition going back to \cite{KellerVossieck} and \cite[Def.~2.1]{AiharaIyama}.
			\begin{dfn}
				A \emph{classical silting collection} in a triangulated category $\cat{C}$ is a set $\silting$ of pairwise non-isomorphic objects of $\cat{C}$ such that
				\begin{itemize}
					\item $\Kar_\cat{C}(\silting)$ is Krull--Schmidt,
					\item objects in $\silting$ are indecomposable,
					\item $\Hom_\cat{C}(P,P'[m])=0$ for all $P,P'\in\silting$, $m>0$,
					\item $\cat{C}=\thick_\cat{C}(\silting)$.
				\end{itemize}
			\end{dfn}
			The only difference to \cref{siltingdef} is that silting collections by definition provide \textbf{t}-structures, while classical silting collections have to generate $\cat{C}$ as thick subcategory.
			
			The following lemma describes the relation between silting collections and classical silting collections in compactly generated triangulated categories.
			We write $\cat{D}^\comp$ for the full subcategory of compact objects of a triangulated category $\cat{D}$ with small coproducts.
			\begin{lemma}
				\label{compactsilting}
				Let $\cat{D}$ be a compactly generated triangulated category.
				Then a set of objects $\silting$ is a compact silting collection in $\cat{D}$ if and only if $\silting$ is a classical silting collection in $\cat{D}^\comp$.
			\end{lemma}
			\begin{proof}
				By \cite[Cor.~4.7]{AiharaIyama} a classical silting collection $\silting$ in $\cat{D}^\comp$ provides a \textbf{t}-structure $t_\silting=(\silting^{\perp_{>0}},\silting^{\perp_{<0}})$ on $\cat{D}$ and hence is a silting collection in $\cat{D}$.
				
				Conversely, if $\silting$ is a compact silting collection in $\cat{D}$, then $\silting$ weakly generates $\cat{D}$ by \cref{siltingtnondeg}, and it follows from general facts (see e.g.~\cite[Prop.~3.4.15]{KrauseHomologicalTheory}) that $\thick_\cat{D}(\silting)=\cat{D}^\comp$.
				Thus $\silting$ is a classical silting collection in $\cat{D}^\comp$.
			\end{proof}
			One often considers classical silting collections in $\Kb(\rprojfg A)$ for a finite-dimensional algebra $A$.
			For instance \cite{KoenigYang} describes the relation of classical silting collections in $\Kb(\rprojfg A)$ to \textbf{t}-structures on $\Db(\rmodfd A)$.
			We would like to rephrase these results using silting collections instead of classical silting collections.
			However, we can't apply \cref{compactsilting} directly, as $\Db(\rmodfd A)$ is not compactly generated since it does not have small coproducts.
			\begin{prop}
				\label{siltingforfdalg}
				Let $A$ be a finite-dimensional algebra and $\silting$ be a set of objects of $\cat{D}=\D(\rMod A)$.
				Then the following are equivalent:
				\begin{tfae}
					\item $\silting$ is a classical silting collection in $\Kb(\rprojfg A)$.
					\item $\silting$ is a compact silting collection in $\D(\rMod A)$.
					\item $\silting$ is a silting collection in $\Db(\rmodfd A)$ and $t_\silting$ is a bounded \textbf{t}-structure on $\Db(\rmodfd A)$.
					\item $\silting$ is a silting collection in $\Db(\rmodfd A)$ and $\thick_\cat{D}(\silting)=\Kb(\rprojfg A)$.
				\end{tfae}
			\end{prop}
			\begin{proof}
				I)$\iff$II): It is well-known that $\cat{D}=\D(\rMod A)$ is compactly generated, and $\cat{D}^\comp=\Kb(\rprojfg A)$.
				Thus by \cref{compactsilting} classical silting collections in $\Kb(\rprojfg A)$ are the same as compact silting collections in $\D(\rMod A)$.
				
				I)$\imp$III): It follows from (the proof of) \cite[Lemma~5.3]{KoenigYang} that classical silting collections $\silting$ in $\Kb(\rprojfg A)$ are silting collections $\silting$ in $\Db(\rmodfd A)$, and that $t_\silting$ is bounded.
				
				III)$\imp$IV): Let $L$ be a simple $A$-module.
				If $t_\silting$ is bounded, then for $P\in\silting$ we have $\Hom_{\Db(\rmodfd A)}(P,L[m])=0$ for $m\gg 0$ or $m\ll 0$, which implies $\silting\subseteq\Kb(\rprojfg A)$.
				It then follows from the proof of \cite[Cor.~6.9]{AdachiMizunoYang} that $\thick_\cat{D}(\silting)=\Kb(\rprojfg A)$.
					
				IV)$\imp$I): This is immediate from \cref{siltinghomnegativity}.
			\end{proof}
			\begin{rem}
				It seems very likely that every silting \textbf{t}-structure on $\Db(\rmodfd A)$ is bounded, or that (equivalently) every silting collection of $\Db(\rmodfd A)$ lies in $\Kb(\rprojfg A)$.
				If this is the case, then both III) and IV) in \cref{siltingforfdalg} reduce to $\silting$ being a silting collection in $\Db(\rmodfd A)$.
			\end{rem}
			The definition of classical silting collections is reminiscent of the properties of indecomposable objects in the coheart of a weight structure.
			Indeed, this is not a coincidence.
			Using silting collections instead of classical silting collections, we obtain:
			\begin{prop}
				\label{weighttosilting}
				Let $\cat{C}\subseteq\cat{D}$ be a thick subcategory of a triangulated category.
				\begin{enumerate}
					\item Let $\silting$ be a silting collection in $\cat{D}$ such that $\thick_\cat{D}(\silting)=\cat{C}$.
						Then $w=(\cat{C}_{w\leq 0},\cat{C}_{w\geq 0})$ with
						\begin{align*}
							\cat{C}_{w\leq 0}&=\Kar_\cat{C}\extclos_\cat{C}\{P[m]\mid P\in\silting, m\geq 0\},\\
							\cat{C}_{w\geq 0}&=\Kar_\cat{C}\extclos_\cat{C}\{P[m]\mid P\in\silting, m\leq 0\}
						\end{align*}
						is a bounded weight structure on $\cat{C}$, and $\silting$ is a full set of isomorphism representatives of the indecomposable objects in $\coheart_w$.
					\item Let $w$ be a bounded weight structure on $\cat{C}$ such that $(\coheart_w^{\perp_{>0}},\coheart_w^{\perp_{<0}})$ is a \textbf{t}-structure on $\cat{D}$ and $\coheart_w$ is Krull--Schmidt.
						Then a full set $\silting$ of isomorphism representatives of the indecomposable objects in $\coheart_w$ is a silting collection in $\cat{D}$.
				\end{enumerate}
			\end{prop}
			\begin{proof}
				The first part is \cite[Thm.~4.3.2]{Bondarko}.
				For the second part, note that since $\coheart_w$ is Krull--Schmidt, $(\silting^{\perp_{>0}},\silting^{\perp_{<0}})=(\coheart_w^{\perp_{>0}},\coheart_w^{\perp_{>0}})$ is a \textbf{t}-structure on $\cat{D}$, and the remaining axioms from \cref{siltingdef} are clear.\qedhere
			\end{proof}
			\begin{rem}
				\label{weighttoclassicalsilting}
				Under the bijection from \cref{weighttosilting}, finite silting collections correspond to weight structures such that the coheart contains finitely many indecomposable objects (up to isomorphism).
				Moreover, \cref{weighttosilting} remains valid if one uses classical silting collections instead of silting collections and leaves out the assumption that the coheart defines a \textbf{t}-structure.
				This version is commonly used, for instance it occurs in \cite{KoenigYang}.
			\end{rem}
			In the setup of \cref{weighttosilting} we would like to know when the coheart of a bounded weight structure on a thick subcategory $\cat{C}\subseteq\cat{D}$ weakly generates $\cat{D}$.
			The following criterion is proved analogously to \cref{siltingtnondeg}.
			\begin{cor}
				Let $\cat{C}\subseteq\cat{D}$ be a thick subcategory of a triangulated category and $w$ be a bounded weight structure on $\cat{C}$.
				If $(\coheart_w^{\perp_{>0}},\coheart_w^{\perp_{<0}})$ defines a \textbf{t}-structure on $\cat{D}$, then $\coheart_w$ weakly generates $\cat{D}$.
			\end{cor}
		\subsection{Adjacency and orthogonality}
			By definition, a silting collection $\silting$ in a triangulated category $\cat{D}$ defines a \textbf{t}-structure $t=(\silting^{\perp_{>0}},\silting^{\perp_{<0}})$ on $\cat{D}$.
			On the other hand, by \cref{weighttosilting} $\silting$ also defines a weight structure $w$ on $\cat{C}=\thick_\cat{D}(\silting)$.
			From the definition of $t$ it is clear that $\cat{D}^{t\leq 0}=(\cat{C}_{w>0})^\perp$ and $\cat{D}^{t\geq 0}=(\cat{C}_{w<0})^\perp$.
			If moreover $\cat{C}=\cat{D}$, then even $\cat{D}^{t\leq 0}=\cat{C}_{w\leq 0}$.
			These relations are described, and generalized by, the notions of orthogonality and adjacency between weight structures and \textbf{t}-structures.
			
			Let $\cat{C}$ and $\cat{D}$ be triangulated categories and $\cat{A}$ an abelian category.
			Following \cite[Def.~5.2.1]{BondarkoWeightTAndBack}, by \emph{duality} we mean a biadditive bifunctor $\Phi\: \cat{C}\times\cat{D}\to\cat{A}$ which is contravariant and cohomological in the first argument, covariant and homological in the second argument, and comes with a natural isomorphism $\Phi(-,-)\cong\Phi(-[1],-[1])$.
			
			Most of the time both $\cat{C}$ and $\cat{D}$ will be subcategories of a triangulated category $\cat{T}$ and $\Phi=\Hom_\cat{T}(-,-)\: \cat{C}\times\cat{D}\to\rMod\groundfield$.
			For sets of objects $\mathcal{X}\subseteq\cat{C}$ and $\mathcal{Y}\subseteq\cat{D}$ we write $\mathcal{X}\perp_\Phi\mathcal{Y}$ if $\Phi(X,Y)=0$ for all $X\in\mathcal{X}$ and $Y\in\mathcal{Y}$, and we define
			\begin{align*}
				\mathcal{X}^{\perp_\Phi}&=\{Y\in\cat{D}\mid \Phi(X,Y)=0\;\forall X\in\mathcal{X}\},&
				\prescript{\perp_\Phi}{}{\mathcal{Y}}&=\{X\in\cat{C}\mid \Phi(X,Y)=0\;\forall Y\in\mathcal{Y}\}.
			\end{align*}
			The following definition is based on \cite[Def.~5.2.1]{BondarkoWeightTAndBack}.
			\begin{dfn}
				Let $w$ be a weight structure on $\cat{C}$ and $t$ a \textbf{t}-structure on $\cat{D}$.
				\begin{itemize}
					\item $w$ is \emph{left orthogonal (with respect to $\Phi$)} to $t$ if $\cat{C}_{w\geq 0}\perp_\Phi\cat{D}^{t<0}$ and $\cat{C}_{w\leq 0}\perp_\Phi\cat{D}^{t>0}$.
					\item The orthogonality is \emph{\textbf{w}-strict} if $\cat{C}_{w\geq 0}=\prescript{\perp_\Phi}{}{(\cat{D}^{t<0})}$ and $\cat{C}_{w\leq 0}=\prescript{\perp_\Phi}{}{(\cat{D}^{t>0})}$.
					\item The orthogonality is \emph{\textbf{t}-strict} if $\cat{D}^{t<0}=(\cat{C}_{w\geq 0})^{\perp_\Phi}$ and $\cat{D}^{t>0}=(\cat{C}_{w\leq 0})^{\perp_\Phi}$.
					\item The orthogonality is \emph{\textbf{w}-\textbf{t}-strict} if it is both \textbf{w}-strict and \textbf{t}-strict.
				\end{itemize}
			\end{dfn}
			If both $\cat{C}$ and $\cat{D}$ are subcategories of a triangulated category $\cat{T}$, then any orthogonality will be with respect to $\Phi=\Hom_\cat{T}(-,-)$ unless explicitly mentioned.
			If moreover $\cat{C}=\cat{D}$, then left orthogonality is also called \emph{left adjacency}.
			
			In \cite{BondarkoWeightTAndBack} only orthogonality and \textbf{t}-strict orthogonality are considered, and there \textbf{t}-strict orthogonality is just called strict orthogonality.
			\begin{rem}
				Note that \cref{weighttosilting} establishes a bijection between silting collections and bounded weight structures that are \textbf{t}-strictly left orthogonal to a \textbf{t}-structure.
			\end{rem}
			If $\cat{C}\subseteq\cat{D}$, then it is possible to characterize left orthogonality in terms of the negative and positive part, and moreover orthogonality and \textbf{w}-strict orthogonality coincide.
			The non-obvious implication I)$\imp$III) of the following statement is already shown in \cite[Prop.~5.2.3]{BondarkoWeightTAndBack}.
			\begin{lemma}
				\label{orthogonalitylemma}
				Let $\cat{C}\subseteq\cat{D}$ be a thick subcategory, $w$ a weight structure on $\cat{C}$ and $t$ a \textbf{t}-structure on $\cat{D}$.
				Then the following are equivalent:
				\begin{tfae}
					\item\label{leftorth} $w$ is left orthogonal to $t$.
					\item\label{strictleftorth} $w$ is \textbf{w}-strictly left orthogonal to $t$.
					\item\label{intersection} $\cat{C}_{w\leq 0}=\cat{D}^{t\leq 0}\cap\cat{C}$ and $\cat{C}_{w\geq 0}=\prescript{\perp}{}{(\cat{D}^{t<0})}\cap\cat{C}$.
				\end{tfae}
			\end{lemma}
			\begin{proof}
				\cref{strictleftorth}$\imp$\cref{leftorth} is trivial.

				\cref{leftorth}$\imp$\cref{intersection}: From $\cat{C}_{w\leq 0}\perp\cat{D}^{t>0}$ it is clear that $\cat{C}_{w\leq 0}\subseteq \prescript{\perp}{}{(\cat{D}^{t>0})}\cap\cat{C}=\cat{D}^{t\leq 0}\cap\cat{C}$.
				The converse inclusion follows from $\cat{C}_{w>0}\perp\cat{D}^{t\leq 0}$ and $\cat{C}_{w\leq 0}=(\cat{C}_{w>0})^\perp$.
				
				For $\cat{C}_{w\geq 0}$, we have by assumption $\cat{C}_{w\geq 0}\subseteq \prescript{\perp}{}{(\cat{D}^{t<0})}\cap\cat{C}$, and from $\cat{C}_{w\leq 0}=\cat{D}^{t\leq 0}\cap\cat{C}$ we get $\prescript{\perp}{}{(\cat{D}^{t<0})}\cap\cat{C}\subseteq\prescript{\perp}{}(\cat{D}^{t<0}\cap\cat{C})=\prescript{\perp}{}{(\cat{C}_{w<0})}=\cat{C}_{w\geq 0}$.
				
				\cref{intersection}$\imp$\cref{strictleftorth}: This is obvious from $\cat{D}^{t\leq 0}=\prescript{\perp}{}{(\cat{D}^{t>0})}$ and the assumptions.
			\end{proof}
			\begin{cor}
				\label{siltingstrictlyorthogonal}
				Let $\silting$ be a silting collection in $\cat{D}$, $t$ its associated silting \textbf{t}-structure and $w$ the induced weight structure on $\thick_\cat{D}(\silting)$.
				Then $w$ is \textbf{w}-\textbf{t}-strictly left orthogonal to $t$.
			\end{cor}
			\begin{proof}
				From the construction of $t$ it is clear that $w$ is \textbf{t}-strictly left orthogonal to $t$, and the orthogonality is \textbf{w}-strict by \cref{orthogonalitylemma}.
			\end{proof}
			The following lemma shows that in the case of adjacent weight structures and \textbf{t}-structures we do not need to distinguish between the various levels of strictness of orthogonality at all.
			The equivalence \cref{adjacency}$\iff$\cref{leftortho}, which recovers the original definition \cite[Def.~4.4.1]{Bondarko} of adjacency, is also shown in \cite[Prop.~1.3.3]{BondarkoWeightTAndBack}.
			\begin{lemma}
				\label{adjacencylemma}
				Let $t$ be a \textbf{t}-structure and $w$ a weight structure on $\cat{C}$.
				Then the following are equivalent:
				\begin{tfae}
					\item\label{adjacency} $\cat{C}^{t\leq 0}=\cat{C}_{w\leq 0}$,
					\item\label{leftortho} $w$ is left orthogonal to $t$,
					\item\label{tstrict} $w$ is \textbf{w}-strictly left orthogonal to $t$,
					\item\label{wstrict} $w$ is \textbf{t}-strictly left orthogonal to $t$,
					\item\label{wtstrict} $w$ is \textbf{w}-\textbf{t}-strictly left orthogonal to $t$.
				\end{tfae}
			\end{lemma}
			\begin{proof}
				\cref{leftortho}$\imp$\cref{adjacency} follows from \cref{orthogonalitylemma}, and the implications \cref{wtstrict}$\imp$\cref{wstrict}, \cref{wtstrict}$\imp$\cref{tstrict}, \cref{tstrict}$\imp$\cref{leftortho} and \cref{wstrict}$\imp$\cref{leftortho} are obvious from the definitions.
				For the remaining implication \cref{adjacency}$\imp$\cref{wtstrict} observe that
				\begin{align*}
					\cat{C}_{w\geq 0}&=\prescript{\perp}{}{(\cat{C}_{w<0})}=\prescript{\perp}{}{(\cat{C}^{t<0})},&
					\cat{C}_{w\leq 0}&=\cat{C}^{t\leq 0}=\prescript{\perp}{}{(\cat{C}^{t>0})},\\
					\cat{C}^{t\leq 0}&=\cat{C}_{w\leq 0}=(\cat{C}_{w>0})^\perp,&
					\cat{C}^{t\geq 0}&=(\cat{C}^{t<0})^\perp=(\cat{C}_{w<0})^\perp.\qedhere
				\end{align*}
			\end{proof}
	\section{Silting collections and derived projectives}\label{derivedprojectives}
		It is well-known that silting collections behave very similar to projective objects.
		To make this precise, in this section we introduce derived projective covers, and show that under some assumptions the derived projective covers of simple objects of the heart are the same as a silting collection.
		As an application, we use derived projective covers to formulate criteria for orthogonality.
		\subsection{Derived projective objects}
			We begin by showing some basic facts about derived projective objects.
			Let $\cat{D}$ be a Krull--Schmidt triangulated category and $t$ a \textbf{t}-structure on $\cat{D}$.
			\begin{dfn}
				An object $P\in\cat{D}$ is \emph{derived projective} (with respect to $t$) if $P\in\cat{D}^{t\leq 0}$ and $\Hom_\cat{D}(P,X[1])=0$ for all $X\in\cat{D}^{t\leq 0}$.
				We write $\DProj_t(\cat{D})$ for the full subcategory of derived projective objects with respect to $t$.
			\end{dfn}
			\cite[Def.~6.1]{GenoveseLowenvdBergh} gives a different definition of derived projective objects, which is equivalent to the above by \cite[Prop.~2.3.5]{GenoveseRamos}.
			Derived projective objects are also known as \emph{Ext-projectives} or just \emph{projectives}, see for instance \cite{CoelhoSimoesPauksztelloPloog}, \cite[\S 7.2.2]{LurieHigherAlgebra} and \cite[\S C.5.7]{LurieSpectralAlgebraicGeometry}.
			
			The definition of derived projectives is motivated by the well-known fact that an object $P$ of an abelian category $\cat{A}$ is projective if and only if $\Ext_\cat{A}^1(P,X)=0$ for all $X\in\cat{A}$.
			From this point of view, the following lemma is an analog of the statement that $\Hom_\cat{A}(P,-)$ is right exact if $P\in\cat{A}$ is projective.
			\begin{lemma}
				\label{derivedprojtruncationfullyfaithful}
				Let $P\in\DProj_t(\cat{D})$ and $X,Y\in\cat{D}$.
				\begin{enumerate}
					\item For $f\: X\to Y$ with $\cone(f)\in\cat{D}^{t<0}$, the map $\Hom_\cat{D}(P,f)\: \Hom_\cat{D}(P,X)\to\Hom_\cat{D}(P,Y)$ is surjective.
					\item $t_{\geq 0}$ and $t_{\leq 0}$ induce isomorphisms
						\begin{equation*}
							\Hom_\cat{D}(P,X)\cong\Hom_{\cat{D}^{t\geq 0}}(t_{\geq 0}P,t_{\geq 0}X)\cong\Hom_{\heart_t}(H_t^0(P),H_t^0(X)).
						\end{equation*}
					\item $P$ is indecomposable if and only if $H_t^0(P)$ is.
				\end{enumerate}
			\end{lemma}
			\begin{proof}\leavevmode
				\begin{enumerate}
					\item This is immediate from the long exact sequence obtained by applying $\Hom_\cat{D}(P,-)$ to the triangle $X\to Y\to\cone(f)\to X[1]$.
					\item From the long exact sequence obtained by applying $\Hom_\cat{D}(P,-)$ to the triangle $t_{<0}X\to X\to t_{\geq 0}X\to t_{<0}X[1]$, and derived projectivity of $P$, we get
						\begin{equation*}
							\Hom_\cat{D}(P,X)\cong\Hom_\cat{D}(P,t_{\geq 0}X)\cong\Hom_{\cat{D}^{t\geq 0}}(t_{\geq 0}P,t_{\geq 0}X).
						\end{equation*}
						This isomorphism is given by the functor $t_{\geq 0}$.
						The second isomorphism follows since $t_{\leq 0}$ is right adjoint to $\cat{D}^{t\leq 0}\hookto\cat{D}$, using that $t_{\geq 0}P=H_t^0(P)$.
					\item By 2) we have $\End_\cat{D}(P)\cong\End_{\heart_t}(H_t^0(P))$.
						Since $\cat{D}$ is Krull--Schmidt, $P$ is indecomposable if and only if $\End_\cat{D}(P)$ is local (and analogously for $H_t^0(P)$).\qedhere
				\end{enumerate}
			\end{proof}
			An easy but important observation is that silting collections consist of derived projectives with respect to their associated silting \textbf{t}-structures, see \cref{siltingderivedproj} below.
			With this in mind, \cref{derivedprojtruncationfullyfaithful} as well as the following lemma is contained in \cite[Prop.~2.5]{AHLSV}.
			Variants of \cref{truncatderivedproj} have already appeared several times in the literature, see for instance \cite{AlNofayee} or \cite[Prop.~4.3]{PsaroudakisVitoria}.
			\begin{lemma}
				\label{truncatderivedproj}
				If $P\in\DProj_t(\cat{D})$, then $t_{\geq 0}P=H_t^0(P)$ is projective in $\heart_t$.
			\end{lemma}
			\begin{proof}
				Since $P\in\cat{D}^{t\leq 0}$ we obviously have $t_{\geq 0}P\in\heart_t$.
				It is well-known (see e.g.~\cite[Prop.~A.7.18]{Achar}) that $\Ext_{\heart_t}^1(t_{\geq 0}P,X)\cong\Hom_\cat{D}(t_{\geq 0}P,X[1])$ for $X\in\heart_t$, where $\Ext_{\heart_t}^1$ is defined via equivalence classes of short exact sequences (Yoneda ext).
				From the long exact sequence obtained by applying $\Hom_\cat{D}(-,X[1])$ to the triangle $t_{<0}P\to P\to t_{\geq 0}P\to t_{<0}P[1]$ and derived projectivity of $P$ it follows that $\Hom_\cat{D}(t_{\geq 0}P,X[1])=0$, and hence $t_{\geq 0}P$ is projective in $\heart_t$.
			\end{proof}
			\begin{cor}
				\label{derprojsplit}
				For $f\: P\to P'$ with $P,P'\in\DProj_t(\cat{D})$ the following are equivalent:
				\begin{tfae}
					\item $f$ is a split epimorphism.
					\item $\cone(f)\in\cat{D}^{t<0}$.
					\item $t_{\geq 0}f=H_t^0(f)$ is an epimorphism in $\heart_t$.
					\item $t_{\geq 0}f=H_t^0(f)$ is a split epimorphism in $\heart_t$.
				\end{tfae}
			\end{cor}
			We will often need to assume that all projectives in $\heart_t$ are obtained as truncations of derived projectives.
			More precisely, we use the following definition from \cite[Def.~6.1 and Def.~6.6]{GenoveseLowenvdBergh}:
			\begin{dfn}
				$\cat{D}$ \emph{has derived projectives} (with respect to $t$) if for every projective $P\in\heart_t$ there is $\hat{P}\in\DProj_t(\cat{D})$ with $H_t^0(\hat{P})\cong P$.
				If moreover $\heart_t$ has enough projectives, we say that $\cat{D}$ \emph{has enough derived projectives} (with respect to $t$).
			\end{dfn}
			In \cref{siltingderprojconditions} we will show that if $\heart_t$ is finite-length, then $\cat{D}$ has enough derived projectives if and only if $t$ is silting.
			In general, $\cat{D}$ does not necessarily have enough derived projectives, even if $\heart_t$ has enough projectives.
			For instance, this is the case for the standard \textbf{t}-structure on $\Dfd(A)$ if $A$ is a non-positive dg algebra such that $H^n(A)$ is finite-dimensional for all $n\in\ints$, but $H^*(A)$ is not, see \cref{notenoughderproj} below.
			\begin{cor}
				\label{equivalencederivedprojheartproj}
				If $\cat{D}$ has derived projectives with respect to $t$, then $t_{\geq 0}=H_t^0\: \DProj_t(\cat{D})\to\Proj(\heart_t)$ is an equivalence of categories.
			\end{cor}
			\begin{proof}
				The functor is well-defined by \cref{truncatderivedproj} and fully faithful by \cref{derivedprojtruncationfullyfaithful}, and that $\cat{D}$ has derived projectives ensures that it is dense.
			\end{proof}
			The following theorem shows that the definition of enough derived projectives given in \cite[Def.~2.2]{CoelhoSimoesPauksztelloPloog} is equivalent to the one we use.
			\begin{thm}
				\label{enoughderivedproj}
				$\cat{D}$ has enough derived projectives with respect to $t$ if and only if $\DProj_t(\cat{D})$ is contravariantly finite in $\cat{D}^{t\leq 0}$ and $\DProj_t(\cat{D})^\perp\cap\heart_t=\{0\}$.
			\end{thm}
			\begin{proof}
				``$\imp$'': For $X\in\cat{D}^{t\leq 0}$ we have $H_t^0(X)\in\heart_t$.
				Since $\cat{D}$ has enough derived projectives, there is an epimorphism $\pi\: P\to H_t^0(X)$ with $P$ projective in $\heart_t$, and moreover $\hat{P}\in\DProj_t(\cat{D})$ with $H_t^0(\hat{P})=t_{\geq 0}\hat{P}\cong P$.
				By \cref{derivedprojtruncationfullyfaithful} there is a unique morphism $\hat{\pi}\colon \hat{P}\to X$ such that $H_t^0(\hat{\pi})=\pi$.
				We claim that $\hat{P}$ is a right $\DProj_t(\cat{D})$-approximation.
				Indeed, for $P'\in\DProj_t(\cat{D})$ by \cref{derivedprojtruncationfullyfaithful} we get a commutative diagram
				\begin{equation*}
					\begin{tikzcd}[column sep=huge]
						\Hom_\cat{D}(P',\hat{P})\arrow{r}{\Hom_\cat{D}(P',\hat{\pi})}\arrow{d}{H_t^0}[swap]{\cong}
						&\Hom_\cat{D}(P',X)\arrow{d}{\cong}[swap]{H_t^0}\\
						\Hom_{\heart_t}(H_t^0(P'),P)\arrow{r}{\Hom_{\heart_t}(H_t^0(P'),\pi)}
						&\Hom_{\heart_t}(H_t^0(P'),H_t^0(X))\rlap{,}
					\end{tikzcd}
				\end{equation*}
				and the bottom map is surjective since $\pi\: H_t^0(\hat{P})\to H_t^0(X)$ is an epimorphism and $H_t^0(P')\in\Proj(\heart_t)$ by \cref{truncatderivedproj}.
				
				Since $\heart_t$ has enough projectives, $X\in\heart_t$ is zero if and only if $\Hom_{\heart_t}(P,X)=0$ for all $P\in\Proj(\heart_t)$.
				Since $\cat{D}$ has derived projectives, for every $P\in\Proj(\heart_t)$ there is $\hat{P}\in\DProj_t(\cat{D})$ with $H_t^0(\hat{P})\cong P$.
				By \cref{derivedprojtruncationfullyfaithful} it follows that $X\in\heart_t$ is zero if and only if $\Hom_\cat{D}(\hat{P},X)=0$ for all $\hat{P}\in\DProj_t(\cat{D})$, as required.
				
				``$\pmi$'': We first show that $\heart_t$ has enough projectives.
				If $X\in\heart_t$, then $X\in\cat{D}^{t\leq 0}$.
				By assumption, there is a right $\DProj_t(\cat{D})$-approximation $\pi\: P\to X$.
				By \cref{truncatderivedproj} $H_t^0(P)$ is projective in $\heart_t$, and so it suffices to show that $H_t^0(\pi)\: H_t^0(P)\to H_t^0(X)=X$ is an epimorphism.
				
				For this, let $P'\in\DProj_t(\cat{D})$ and apply $\Hom_\cat{D}(P',-)$ to the triangle $P\longto{\pi}X\to \cone(\pi)\to P[1]$.
				This gives an exact sequence
				\begin{equation*}
					\Hom_\cat{D}(P',P)\to\Hom_\cat{D}(P',X)\to\Hom_\cat{D}(P',\cone(\pi))\to\Hom_\cat{D}(P',P[1]).
				\end{equation*}
				The first map is surjective since $\pi\: P\to X$ is a right $\DProj_t(\cat{D})$-approximation, and the last term vanishes as $P'$ is derived projective and $P[1]\in\cat{D}^{t<0}$.
				Thus $\Hom_\cat{D}(P',\cone(\pi))=0$.
				As $t_{\geq 0}$ is left adjoint to $\cat{D}^{t\geq 0}\hookto\cat{D}$ and $t_{\geq 0}P'=H_t^0(P')$, we get
				\begin{equation*}
					\begin{split}
						\Hom_{\cat{D}}(P',H_t^0(\cone(\pi)))&\cong\Hom_{\heart_t}(H_t^0(P'),H_t^0(\cone(\pi))\\
						&\cong\Hom_\cat{D}(P',\cone(\pi))=0,
					\end{split}
				\end{equation*}
				where the last isomorphism is by \cref{derivedprojtruncationfullyfaithful}.
				Thus $H_t^0(\cone(\pi))\in\DProj_t(\cat{D})^\perp\cap\heart_t=\{0\}$.
				
				To show that $\cat{D}$ has derived projectives, let $P\in\heart_t$ be projective and let $\pi\: \tilde{P}\to P$ be a right $\DProj_t(\cat{D})$-approximation.
				By the previous argument, we have $\cone(\pi)\in\cat{D}^{t<0}$, and thus we get an epimorphism $H_t^0(\tilde{P})\to P$ in $\heart_t$.
				This splits since $P$ is projective, and thus $P$ is a summand of $H_t^0(\tilde{P})$.
				Since $H_t^0\: \DProj_t(\cat{D})\to\Proj(\heart_t)$ is fully faithful, there must be a corresponding summand $\hat{P}$ of $\tilde{P}$ with $H_t^0(\hat{P})\cong P$.
			\end{proof}
			\begin{rem}
				As is explained in \cite[Rem.~2.3]{CoelhoSimoesPauksztelloPloog}, in \cref{enoughderivedproj} the assumption that $\DProj_t(\cat{D})$ is contravariantly finite is unnecessary if $\DProj_t(\cat{D})$ contains only finitely many indecomposables.
			\end{rem}
			By combining \cref{enoughderivedproj} with \cite[Thm.~2.4]{CoelhoSimoesPauksztelloPloog} (see also \cite[Thm.~5.3.1]{BondarkoWeightTAndBack}) we obtain the following criterion for the existence of a weight structure that is left adjacent to a given \textbf{t}-structure.
			\begin{cor}
				\label{derprojadjacency}
				For a bounded above \textbf{t}-structure $t$ on a Hom-finite Krull--Schmidt triangulated category $\cat{D}$ the following are equivalent:
				\begin{tfae}
					\item\label{adjacencycond1} $\DProj_t(\cat{D})$ is contravariantly finite in $\cat{D}^{t\leq 0}$ and $\DProj_t(\cat{D})^\perp\cap\heart_t=\{0\}$.
					\item\label{adjacencycond2} $\cat{D}$ has enough derived projectives with respect to $t$.
					\item\label{adjacencycond3} $t$ admits a left adjacent weight structure.
				\end{tfae}
				Moreover, if these conditions hold, then $\heart_t$ is covariantly finite in $\cat{D}$.
			\end{cor}
			\begin{proof}
				By \cref{enoughderivedproj}, \cref{adjacencycond1} is equivalent to \cref{adjacencycond2}.
				Moreover \cref{adjacencycond2} is equivalent to \cite[Thm.~2.4 (2)]{CoelhoSimoesPauksztelloPloog} by \cref{equivalencederivedprojheartproj}, and \cref{adjacencycond1} is \cite[Thm.~2.4 (1)]{CoelhoSimoesPauksztelloPloog} without the assumption that $\heart_t$ is covariantly finite in $\cat{D}$.
				Thus the remaining implications follow from \cite[Thm.~2.4 and Rem.~2.5]{CoelhoSimoesPauksztelloPloog}.
			\end{proof}
			\begin{ex}\label{surprisingweightstruct}
				In particular, \cref{derprojadjacency} shows that for a finite-dimensional algebra $A$, the standard \textbf{t}-structure on $\Db(\rmodfd A)$ (see \cite[Ex.~1.3.2]{BBD}) admits a left adjacent weight structure.
				Indeed, in this case \cref{adjacencycond2} is obviously satisfied: the projective generator $A$ of $\heart_t\cong\rmodfd A$ is derived projective since $\Hom_{\Db(\rmodfd A)}(A,X[n])\cong H^n(X)$ for all $n\in\ints$ (or, in other words, since $A$ is the silting object defining the standard \textbf{t}-structure).
				An alternative way to obtain this weight structure on $\Db(\rmodfd A)$ is via \cite[Lemma~4.10]{AdachiMizunoYang} and \cref{wtequalsst} below.
				
				The adjacent weight structure can also be described explicitly, as follows.
				Let $\cat{D}\subseteq\D^-(\rmodfd A)$ be the full triangulated subcategory of complexes with finite-dimensional total cohomology, and $\cat{C}\subseteq\K^-(\rprojfg A)$ the full subcategory of complexes with finite-dimensional total cohomology.
				The obvious inclusions
				\begin{equation*}
					\begin{tikzcd}
						\cat{C}\arrow[hook]{r}
						&\cat{D}
						&\Db(\rmodfd A)\arrow[left hook->,swap]{l}
					\end{tikzcd}
				\end{equation*}
				are equivalences since any $Y\in\cat{D}$ can be $t$-truncated to an isomorphic object that lies in $\Db(\rmodfd A)$, and $\cat{C}$ precisely consists of the projective resolutions of objects in $\cat{D}$.
				Note that (by construction) the equivalence $\Db(\rmodfd A)\to\cat{C}$ sends a complex to a projective resolution.
				
				The standard weight structure on $\K^-(\rprojfg A)$ from \cite[\S 1.1]{Bondarko} restricts to a weight structure $w$ on $\cat{C}$, and thus yields a weight structure on $\Db(\rmodfd A)$.
				For this it suffices to check that if $X\in\cat{C}$, then there is a weight decomposition $w_{>0}X\to X\to w_{\leq 0}X\to w_{>0}X[1]$ with $w_{>0}X,w_{\leq 0}X\in\cat{C}$.
				But this is obvious since for the standard weight structure, $w_{>0}X$ and $w_{\leq 0}X$ are given by ``brutal truncation'' of $X$ (note that $X$ is, by definition, a complex of finitely generated projectives).
				
				The weight structure $w$ is left adjacent to the standard \textbf{t}-structure on $\Db(\rmodfd A)$ since $\cat{C}_{w\leq 0}$ precisely consists of the projective resolutions of objects in $\Db(\rmodfd A)^{t\leq 0}$.
				Note that $w$ is always bounded above, but bounded below if and only if $A$ has finite global dimension.
			\end{ex}
		\subsection{Derived projective covers}
			For an object $X\in\cat{D}$ we would like to find a minimal derived projective object approximating $X$.
			This is made precise by the following definition, which is dual to \cite[Ex.~C.5.7.9]{LurieSpectralAlgebraicGeometry}.
			\begin{dfn}
				\label{derprojcoverdef}
				A \emph{derived projective cover} of $X\in\cat{D}$ is a morphism $\pi\: P\to X$ such that $P$ is derived projective and $H_t^0(\pi)\: H_t^0(P)\to H_t^0(X)$ is a projective cover of $H_t^0(X)$ in $\heart_t$.
			\end{dfn}
			\begin{lemma}
				\label{derivedprojcoverunique}
				The derived projective cover of $X\in\cat{D}$ is unique up to isomorphism (if it exists).
			\end{lemma}
			\begin{proof}
				Let $\pi_1\: P_1\to X$ and $\pi_2\: P_2\to X$ be derived projective covers of $X$.
				Then $H_t^0(\pi_1)\: H_t^0(P_1)\to H_t^0(X)$ and $H_t^0(\pi_2)\: H_t^0(P_2)\to H_t^0(X)$ are projective covers of $H_t^0(X)$ in $\heart_t$.
				Since projective covers are unique up to isomorphism, there is an isomorphism $g\: P_1\to P_2$ with $H_t^0(\pi_2)g=H_t^0(\pi_1)$, and by \cref{derivedprojtruncationfullyfaithful} there is an isomorphism $\hat{g}\: P_1\to P_2$ with $\pi_2\hat{g}=\pi_1$.
			\end{proof}
			\begin{lemma}
				Let $P$ be derived projective.
				Then $\pi\: P\to X$ is a derived projective cover of $X$ if and only if $t_{\leq 0}\pi\: P\to t_{\leq 0}X$ is a derived projective cover of $t_{\leq 0}X$.
			\end{lemma}
			\begin{proof}
				This is obvious since $H_t^0\circ t_{\leq 0}=H_t^0$.
			\end{proof}
			Recall that in a Krull--Schmidt abelian category, a morphism $\pi\: P\to X$ is a projective cover if and only if it satisfies one of the following equivalent conditions:
			\begin{tfae}
				\item $P$ is projective, $\pi$ an epimorphism, and for any epimorphism $\pi'\: P'\to X$ with $P'$ projective there is $g\: P'\to P$ with $\pi g=\pi'$, and any such $g$ is an epimorphism.
				\item $\pi\: P\to X$ is a minimal right approximation of $X$ by projectives.
			\end{tfae}
			If $X$ is simple, then moreover $\pi\: P\to X$ is a projective cover if and only if it satisfies one of the following equivalent conditions:
			\begin{tfae}
				\item $P$ is projective and $\pi\neq 0$, and for any non-zero $\pi'\: P'\to X$ with $P'$ projective there is $g\: P'\to P$ with $\pi g=\pi'$, and any such $g$ is the projection onto a direct summand.
				\item $P$ is projective, indecomposable, and $\pi\neq 0$.
			\end{tfae}
			The following lemma provides analogous characterizations of derived projective covers in more specific situations.
			In general, a good strategy to pass from statements about projective objects to statements about derived projective objects is to replace ``$f$ is an epimorphism'' by ``$\cone(f)\in\cat{D}^{t<0}$''.
			This can also be seen in \cref{derivedprojtruncationfullyfaithful} above, which is also the main reason behind this phenomenon. 
			\begin{lemma}
				\label{derprojcoverequivalentdefs}
				Assume that $\cat{D}$ has derived projectives with respect to $t$.
				\begin{enumerate}
					\item If $X\in\cat{D}^{t\leq 0}$, then for $\pi\: P\to X$ the following are equivalent:
						\begin{tfae}
							\item\label{derprojcovercond1} $\pi\: P\to X$ is a derived projective cover of $X$.
							\item\label{derprojcovercond2} $\pi\: P\to X$ satisfies the following conditions:
								\begin{itemize}
									\item $P$ is derived projective,
									\item $\cone(\pi)\in\cat{D}^{t<0}$,
									\item for $\pi'\: P'\to X$ with $P'$ derived projective and $\cone(\pi')\in \cat{D}^{t<0}$ there is $g\: P'\to P$ with $\pi g=\pi'$,
									\item and $\cone(g)\in\cat{D}^{t<0}$ for any such $g$.
								\end{itemize}
							\item\label{derprojcovercond3} $\pi\: P\to X$ is a minimal right approximation of $X$ by derived projective objects.
						\end{tfae}
					\item If $L\in\heart_t$ is simple, then for $\pi\: P\to L$ the following are equivalent:
						\begin{tfae}
							\item\label{derprojcoversimple1} $\pi\: P\to L$ is a derived projective cover of $L$.
							\item\label{derprojcoversimple2} $P$ is indecomposable derived projective and $\pi\neq 0$.
							\item\label{derprojcoversimple3} $\pi\: P\to L$ satisfies the following conditions:
								\begin{itemize}
									\item $P$ is derived projective,
									\item $\pi\neq 0$,
									\item for any non-zero $\pi'\: P'\to L$ with $P'$ derived projective there is $g\: P'\to P$ such that $\pi'=\pi g$,
									\item and any such $g$ is the projection onto a direct summand.
								\end{itemize}
						\end{tfae}
				\end{enumerate}
			\end{lemma}
			\begin{proof}\leavevmode
				\begin{enumerate}
					\item \cref{derprojcovercond1}$\imp$\cref{derprojcovercond2}: Let $X\in\cat{D}^{t\leq 0}$ and $\pi\: P\to X$ be a derived projective cover.
						From the triangle $P\to X\to\cone(\pi)\to P[1]$ we get $\cone(\pi)\in\cat{D}^{t\leq 0}$, and as $H_t^0(\pi)$ is an epimorphism we have $H_t^0(\cone(\pi))=0$, thus $\cone(\pi)\in\cat{D}^{t<0}$.
						Now let $\pi'\: P'\to X$ with $P'$ derived projective and $\cone(\pi')\in\cat{D}^{t<0}$.
						Applying $\Hom_\cat{D}(P',-)$ to the triangle $P\to X\to\cone(\pi)\to P[1]$ and using $\Hom_\cat{D}(P',\cone(\pi))=0$ (since $\cone(\pi)\in\cat{D}^{t<0}$ and $P'$ is derived projective) shows that $\pi'$ factors through $\pi$.
						So let $g\: P'\to P$ be any morphism with $\pi g=\pi'$.
						Since $P,P'\in\cat{D}^{t\leq 0}$ it is clear that $\cone(g)\in\cat{D}^{t\leq 0}$.
						Note that by \cref{truncatderivedproj} $H_t^0(\pi')\: H_t^0(P')\to H_t^0(X)$ is an epimorphism from a projective object in $\heart_t$.
						As $H_t^0(\pi)H_t^0(g)=H_t^0(\pi')$ and $H_t^0(\pi)\: H_t^0(P)\to H_t^0(X)$ is a projective cover it follows that $H_t^0(g)$ must be an epimorphism.
						This means $H_t^0(\cone(g))=0$, and hence $\cone(g)\in\cat{D}^{t<0}$.
						
						\cref{derprojcovercond2}$\imp$\cref{derprojcovercond3}: Let $\pi'\: P'\to X$ be any morphism with $P'$ derived projective.
						The long exact sequence obtained by applying $\Hom_\cat{D}(P',-)$ to the triangle $P\longto{\pi}X\to\cone{\pi}\to P[1]$ shows that $\Hom_\cat{D}(P',\pi)\: \Hom_\cat{D}(P',P)\to\Hom_\cat{D}(P',X)$ is surjective, since $\cone(\pi)\in\cat{D}^{t<0}$ and $P'$ is derived projective.
						Thus $\pi\: P\to X$ is a right approximation.
						
						For minimality, let $g\: P\to P$ with $\pi g=\pi$.
						Then by assumption $\cone(g)\in\cat{D}^{t<0}$, so by \cref{derprojsplit} $g$ is a split epimorphism.
						Since $\cat{D}$ is Krull--Schmidt, it follows that $g$ is an isomorphism.
						
						\cref{derprojcovercond3}$\imp$\cref{derprojcovercond1}: We show that $H_t^0(\pi)\: H_t^0(P)\to H_t^0(X)$ is a minimal right approximation by projectives.
						Let $P'\in\heart_t$ be projective and $\pi'\: P'\to X$.
						By assumption there is a derived projective $\hat{P}'$ with $H_t^0(\hat{P}')\cong P'$, and we get an induced morphism $\hat{\pi}'\: \hat{P}'\to t_{\geq 0}\hat{P}'=H_t^0(\hat{P}')\longto{\pi'}X$ with $H_t^0(\hat{\pi}')=\pi'$.
						\Cref{derivedprojtruncationfullyfaithful} implies that $H_t^0$ induces a bijection between morphisms $\hat{g}\: \hat{P}'\to P$ with $\hat{\pi}'=\pi\hat{g}$ and morphisms $g\: P'=H_t^0(\hat{P}')\to H_t^0(P)$ with $\pi'=H_t^0(\pi)g$, and the claim follows from this.
					\item \cref{derprojcoversimple1}$\iff$\cref{derprojcoversimple2}: If $P$ is derived projective, then $\pi\: P\to L$ is a derived projective cover of $L$ iff $H_t^0(\pi)\: H_t^0(P)\to L$ is a projective cover of $L$ in $\heart_t$ iff $H_t^0(\pi)\neq 0$ and $H_t^0(P)$ is indecomposable projective iff $\pi\neq 0$ and $P$ is indecomposable (by \cref{derivedprojtruncationfullyfaithful}).
						
						\cref{derprojcoversimple1}$\iff$\cref{derprojcoversimple3}: Observe that for $\pi\: P\to L$ we have $\cone(\pi)\in\cat{D}^{t\leq 0}$, and by \cref{derivedprojtruncationfullyfaithful} and simplicity of $L$ we get $\pi\neq 0$ iff $H_t^0(\pi)\neq 0$ iff $H_t^0(\pi)$ is an epimorphism iff $\cone(\pi)\in\cat{D}^{t<0}$.
						Similarly $g\: P\to P'$ satisfies $\cone(g)\in\cat{D}^{t\leq 0}$, and by \cref{derprojsplit} $g$ is the projection onto a direct summand if and only if $\cone(g)\in\cat{D}^{t<0}$.
						Therefore the claim follows from 1).\qedhere
				\end{enumerate}
			\end{proof}
		\subsection{Silting collections as derived projective covers}
			Silting collections provide an important source of derived projective objects.
			\begin{lemma}
				\label{siltingderivedproj}
				Let $\silting$ be a silting collection in $\cat{D}$ and $t$ its associated silting \textbf{t}-structure.
				Then any $P\in\silting$ is derived projective with respect to $t$.
			\end{lemma}
			\begin{proof}
				By definition we have $\cat{D}^{t<0}=\silting^{\perp_{\geq 0}}$, and also $\silting\subseteq\silting^{\perp_{>0}}=\cat{D}^{t\leq 0}$ by \cref{siltinghomnegativity}, which precisely means that $\silting$ consists of derived projective objects.
			\end{proof}
			The following theorem shows that derived projective covers provide a convenient description of the relation between a silting collection and the simple objects in the heart of the associated silting \textbf{t}-structure.
			In particular, this observation can be used to formulate the bijections between simple-minded collections and silting collections from \cite{KoenigYang}, see \cref{weighttbijections} below.
			This result is very similar to \cite[Thm.~2.4]{CoelhoSimoesPauksztelloPloog} and \cite[Thm.~5.3.1 II.]{BondarkoWeightTAndBack}.
			\begin{thm}
				\label{siltingderprojconditions}
				Let $t$ be a non-degenerate \textbf{t}-structure on $\cat{D}$ with finite-length heart.
				Let $\simplem$ be a full set of isomorphism representatives of the simple objects in $\heart_t$ and $\silting$ a full set of isomorphism representatives of the indecomposable derived projectives.
				Then the following are equivalent:
				\begin{tfae}
					\item\label{tsilting} $t$ is silting (and $\silting$ is the silting collection).
					\item\label{tsimplemsiltingrel} There is a bijection $\phi\: \silting\to\simplem$ such that for $P\in\silting$, $L\in\simplem$, $m\in\ints$ we have
						\begin{equation}
							\label{simplemsiltingrel}
							\Hom_\cat{D}(P,L[m])\cong\begin{cases}
								\End_\cat{D}(L)&\text{if $L=\phi(P)$, $m=0$},\\
								0&\text{otherwise}
							\end{cases}
						\end{equation}
						as left $\End_\cat{D}(L)$-modules.
					\item\label{tderprojcover} Every $L\in\simplem$ admits a derived projective cover (and $\silting$ is the set of these derived projective covers).
					\item\label{tenoughderproj} $\cat{D}$ has enough derived projectives with respect to $t$.
				\end{tfae}
			\end{thm}
			\begin{proof}
				\cref{tsilting}$\imp$\cref{tsimplemsiltingrel}: Let $\silting'$ be a silting collection with $t=(\silting'^{\perp_{<0}},\silting'^{\perp_{>0}})$.
				By \cref{siltingtnondeg} $\silting'$ weakly generates $\cat{D}$, and thus for each $L\in\simplem$ there is some $P\in\silting'$ and $m\in\ints$ with $\Hom_\cat{D}(P,L[m])\neq 0$.
				From $L\in\heart_t$ we get $m=0$, so using $P\in\cat{D}^{t\leq 0}$ we get
				\begin{equation*}
					0\neq \Hom_\cat{D}(P,L)=\Hom_{\cat{D}^{t\geq 0}}(t_{\geq 0}P,L)=\Hom_{\heart_t}(H_t^0(P),L).
				\end{equation*}
				Since $L$ is simple in $\heart_t$ it follows that there is an epimorphism $H_t^0(P)\to L$.
				As $P$ is indecomposable, so is $H_t^0(P)$ by \cref{derivedprojtruncationfullyfaithful}, and thus $H_t^0(P)$ is the projective cover of $L$ in $\heart_t$.
				From this it follows that $H_t^0(P)$, and (by \cref{derivedprojtruncationfullyfaithful} again) also $P$, is unique up to isomorphism.
				So we get a bijection $\phi\: \silting'\to\simplem$ by defining $\phi(P)=L$, and moreover
				\begin{equation*}
					\Hom_\cat{D}(P,L)\cong\Hom_{\heart_t}(H_t^0(P),L)\cong\End_{\heart_t}(L)
				\end{equation*}
				as left $\End_{\heart_t}(L)$-module, as desired.
				Finally, since $H_t^0(\silting')$ is a full set of indecomposable projectives in $\heart_t$ and $\silting'$ consists of derived projectives by \cref{siltingderivedproj}, it follows from \cref{equivalencederivedprojheartproj} that $\silting'=\silting$ is the set of indecomposable derived projectives.
				
				\cref{tsimplemsiltingrel}$\imp$\cref{tderprojcover}: Let $L\in\simplem$, $P=\phi^{-1}(L)$, and $\pi\: P\to L$ correspond to $\id_L$.
				Then $\pi\neq 0$, and moreover $H_t^0(P)$ (and thus, by \cref{derivedprojtruncationfullyfaithful}, also $P$) must be indecomposable since otherwise it would admit two simple quotients, which is impossible by \cref{simplemsiltingrel}.
				Thus $\pi\: P\to L$ is the derived projective cover of $L$ by \cref{derprojcoverequivalentdefs}.
				
				\cref{tderprojcover}$\imp$\cref{tenoughderproj}: Since every $L\in\simplem$ has a derived projective cover, it by definition has a projective cover in $\heart_t$.
				As $\heart_t$ is finite-length, it follows that $\heart_t$ has enough projectives.
				Moreover, the projective covers of the simple objects are a full set of isomorphism representatives of the indecomposable projectives in $\heart_t$.
				Thus the indecomposable projectives arise as $t$-truncations of derived projectives, and therefore $\cat{D}$ has enough derived projectives.
				
				\cref{tenoughderproj}$\imp$\cref{tsilting}: By \cref{equivalencederivedprojheartproj}, $H_t^0(\silting)$ is the set of indecomposable projectives of $\heart_t$.
				We claim that $\cat{D}^{t\leq 0}=\silting^{\perp_{>0}}$ and $\cat{D}^{t\geq 0}=\silting^{\perp_{<0}}$.
				For $X\in\cat{D}$, $P\in\silting$ and $n\in\ints$ we get from \cref{derivedprojtruncationfullyfaithful}
				\begin{equation*}
					\Hom_\cat{D}(P,X[n])\cong\Hom_{\heart_t}(H_t^0(P),H_t^0(X[n]))=\Hom_{\heart_t}(H_t^0(P),H_t^n(X)).
				\end{equation*}
				Since $\{H_t^0(P)\mid P\in\silting\}$ is a full set of indecomposable projectives in $\heart_t$ and $\heart_t$ is finite-length, we have $H_t^n(X)=0$ if and only if $\Hom_\cat{D}(P,X[n])=0$ for all $P\in\silting$.
				The claim follows from this since by non-degeneracy of $t$ we know that $X\in\cat{D}^{t\leq 0}$ if and only if $H_t^n(X)=0$ for all $n>0$, and similarly for $\cat{D}^{t\geq 0}$.
			\end{proof}
			Using \cref{siltingderprojconditions} we can now show that not every triangulated category has derived projectives with respect to every \textbf{t}-structure.
			\begin{ex}
				\label{notenoughderproj}
				Let $A$ be a non-positive dg algebra such that $H^n(A)$ is finite-dimensional for all $n\in\ints$ but $H^*(A)$ is not finite-dimensional.
				Then $A$ is a silting object in $\D(A)$ (see e.g.~\cite[Appendix~A]{BruestleYang}), and the silting \textbf{t}-structure restricts to $\Dfd(A)$.
				These \textbf{t}-structures are the \emph{standard \textbf{t}-structures} on $\D(A)$ and $\Dfd(A)$.
				The heart of the standard \textbf{t}-structure on $\Dfd(A)$ is equivalent to $\rmodfd H^0(A)$ and thus has enough projectives.
				However, this \textbf{t}-structure is not silting, and so by \cref{siltingderprojconditions} there is no derived projective that truncates to the projective generator of $\heart_t$.
				
				To see that the standard \textbf{t}-structure is not silting, suppose for a contradiction that $P\in\Dfd(A)$ is a silting object defining the standard \textbf{t}-structure.
				Observe that then $t_{\geq n} P\cong t_{\geq n}A$ for all $n\leq 0$, since both these objects represent the functor $H^0(-)\colon \Dfd(A)^{t\leq 0}\cap\Dfd(A)^{t\geq n}\to\heart_t$ (this uses the equivalences $\rmodfd A\cong\heart_t\cong\rmodfd \End_{\heart_t}(H^0(P))$, and \cref{derivedprojtruncationfullyfaithful}).
				As $P\in\Dfd(A)$, there is $N\leq 0$ with $t_{\geq n}P\cong P$ for all $n\leq N$, and thus we also have $t_{\geq n}A\cong t_{\geq n}P\cong P$ for all $n\leq N$.
				But this implies $H^n(A)\cong H^n(P)=0$ for $n<N$ and thus $A\in\Dfd(A)$, a contradiction.
			\end{ex}
			As an application, we obtain the following criterion for left orthogonality between weight structures defined from silting collections and \textbf{t}-structures defined from simple-minded collections.
			\begin{thm}
				\label{simplemindedsiltingorthogonal}
				Let $\cat{D}$ be a triangulated category, $\silting$ a silting collection in $\cat{D}$ and $\simplem$ a simple-minded collection in $\cat{D}$.
				Let $w$ be the weight structure on $\cat{C}=\thick_\cat{D}(\silting)$ defined by $\silting$ and $t$ the \textbf{t}-structure on $\cat{D}$ defined by $\simplem$.
				Then the following are equivalent:
				\begin{tfae}
					\item\label{negativepart} $w$ is left orthogonal to $t$.
					\item\label{strictorthogonal} $w$ is \textbf{w}-\textbf{t}-strictly left orthogonal to $t$.
					\item\label{tfromsilting} $t=(\silting^{\perp_{>0}},\silting^{\perp_{<0}})$ is the silting \textbf{t}-structure associated with $\silting$.
					\item\label{simplemsiltingderprojcover} There is a bijection $\phi\: \silting\to\simplem$ such that $P\in\silting$ is the derived projective cover of $\phi(\silting)$.
					\item\label{simplemsiltingbij} There is a bijection $\phi\: \silting\to\simplem$ such that for $P\in\silting$, $L\in\simplem$ and $m\in\ints$ we have isomorphisms of left $\End_\cat{D}(L)$-modules
						\begin{equation*}
							\Hom_\cat{D}(P,L[m])\cong\begin{cases}
								\End_\cat{D}(L)&\text{if $L=\phi(P)$, $m=0$},\\
								0&\text{otherwise}.
							\end{cases}
						\end{equation*}
				\end{tfae}
			\end{thm}
			\begin{proof}
				\cref{negativepart}$\imp$\cref{strictorthogonal}: \textbf{w}-strictness follows from \cref{orthogonalitylemma}.
				For \textbf{t}-strictness, note that by \cref{orthogonalitylemma} we have $\cat{C}_{w\leq 0}=\cat{D}^{t\leq 0}\cap\cat{C}$ and $\cat{C}_{w\geq 0}=\prescript{\perp}{}{(\cat{D}^{t<0})}\cap\cat{C}$.
				Therefore
				\begin{equation*}
					\cat{D}^{t\geq 0}=(\cat{D}^{t<0})^\perp\subseteq(\cat{C}_{w<0})^\perp.
				\end{equation*}
				and similarly $\cat{D}^{t\leq 0}\subseteq (\cat{C}_{w>0})^\perp$.
				But by assumption both $t=(\cat{D}^{t\leq 0},\cat{D}^{t\geq 0})$ and $((\cat{C}_{w>0})^\perp,(\cat{C}_{w<0})^\perp)=(\silting^{\perp_{>0}},\silting^{\perp_{<0}})$ are \textbf{t}-structures on $\cat{D}$, and therefore they must agree.
				
				\cref{strictorthogonal}$\imp$\cref{tfromsilting}: This is clear from the construction of $w$ from $\silting$ in \cref{weighttosilting}.
				
				\cref{tfromsilting}$\imp$\cref{simplemsiltingderprojcover}: This is part of \cref{siltingderprojconditions}.
				
				\cref{simplemsiltingderprojcover}$\imp$\cref{simplemsiltingbij}: As $L\in\heart_t$ and $P$ is derived projective, we have $\Hom_\cat{D}(P,L[m])=0$ for $m\neq 0$.
				For $m=0$ we get from \cref{derivedprojtruncationfullyfaithful}
				\begin{equation*}
					\Hom_\cat{D}(P,L)\cong\Hom_{\heart_t}(H_t^0(P),L)\cong\begin{cases}
						\End_{\heart_t}(L)&\text{if }L=\phi(P),\\
						0&\text{else},
					\end{cases}
				\end{equation*}
				since by definition $H_t^0(P)$ is the projective cover of $\phi(P)$ in $\heart_t$.
				
				\cref{simplemsiltingbij}$\imp$\cref{negativepart}:
				From the construction of $w$ and $t$ it follows that left orthogonality is equivalent to $\Hom_\cat{D}(P[m],L[n])=0$ for all $P\in\silting$, $L\in\simplem$, and either $m\leq 0$ and $n>0$, or $m\geq 0$ and $n<0$.
				This condition is obvious from the assumptions.
			\end{proof}
		\subsection{The WT correspondence revisited}
			In many important examples, \textbf{w}-\textbf{t}-strict orthogonality yields a bijection between weight structures and \textbf{t}-structures.
			A unified setup for this is provided by the following definition.
			\begin{dfn}
				\label{wtpairdef}
				Let $\cat{T}$ be an idempotent-complete triangulated category and $\cat{C},\cat{D}\subseteq\cat{T}$ thick subcategories.
				We call $(\cat{C},\cat{D})$ a \emph{WT pair in $\cat{T}$} if there is a weight structure $w$ and a non-degenerate \textbf{t}-structure $t$ on $\cat{T}$ such that
				\begin{enumerate}
					\item $w$ is left adjacent to $t$,
					\item $w$ and $t$ are bounded above,
					\item $\cat{C}=\thick_\cat{T}(\coheart_w)$ and $\cat{D}=\tria_\cat{T}(\heart_t)$,
					\item $\coheart_w$ is Krull--Schmidt with finitely many indecomposables, and $\heart_t$ is Hom-finite finite-length with finitely many simples.
				\end{enumerate}
			\end{dfn}
			In \cref{wtequalsst} below we will show that WT pairs are the same as the \emph{ST pairs} defined in \cite[Def.~4.3]{AdachiMizunoYang}.
			In contrast to that definition, we do not use silting collections and instead define WT pairs via weight structures and \textbf{t}-structures.
			The axioms for WT pairs are similar to the conditions from \cite[Prop.~4.17]{AdachiMizunoYang}.
			\begin{ex}
				\label{wtpairs}
				We list some known examples of WT pairs.
				\begin{enumerate}
					\item Let $A$ be a finite-dimensional algebra.
						Then $(\Kb(\rprojfg A),\Db(\rmodfd A))$ is a WT pair in $\D^-(\rmodfd A)$.
					
						By using the weight structure on $\Db(\rmodfd A)$ described in \cref{surprisingweightstruct}, one sees that $(\Kb(\rprojfg A),\Db(\rmodfd A))$ is also a WT pair in $\Db(\rmodfd A)$, cf.~\cite[Lemma~4.10]{AdachiMizunoYang}.
					\item Let $A$ be a non-positive dg algebra such that $H^n(A)$ is finite-dimensional for all $n\in\ints$.
						Then $(\perf(A),\Dfd(A))$ is a WT pair in $\Dfd^-(A)=\{X\in\D(A)\mid \sum_{k\geq n} \dim H^k(X)<\infty\;\forall n\in\ints\}$ by \cite[Ex.~3.4]{Fushimi}.
					\item Let $A$ be a non-positive dg algebra such that $H^0(A)$ is finite-dimensional and $\Dfd(A)\subseteq\perf(A)$.
						Then $(\perf(A),\Dfd(A))$ is a WT pair in $\perf(A)$ by \cite[Lemma~4.15]{AdachiMizunoYang}.
					\item Let $A$ be a positive dg algebra such that $\Dfd(A)=\thick_{\D(A)}(D(A))$, where $D=\Hom_{\dgrMod\groundfield}(-,\groundfield)$ is the $\groundfield$-linear duality functor.
						Then \cite[Cor.~4.1 and Thm.~7.1]{KellerNicolas} provide a weight structure on $\Dfd(A)$ and (via the equivalence provided by the Nakayama functor) a \textbf{t}-structure on $\thick(D(A))$.
						By taking K-injective resolutions, one sees that these are left adjacent to each other, and it follows that $(\Dfd(A),\thick(D(A)))$ is a WT pair in $\Dfd(A)$.
				\end{enumerate}
				Note that for 1)--3) it is very easy to check the axioms from \cite[Def.~4.3]{AdachiMizunoYang}, but hard to give explicit descriptions of the adjacent weight structure required for \cref{wtpairdef}.
				
				The first example also shows that the ambient triangulated category $\cat{T}$ for a WT pair is in general not unique.
				In fact, as observed in \cite[\S 6.1]{AdachiMizunoYang}, if $(\cat{C},\cat{D})$ is a WT pair in $\cat{T}$, then it is also a WT pair in any thick subcategory $\cat{T}'\subseteq\cat{T}$ containing both $\cat{C}$ and $\cat{D}$.
			\end{ex}
			The following proposition shows that WT pairs are the same as the \emph{ST pairs} defined in \cite[Def.~4.3]{AdachiMizunoYang}.
			\begin{prop}
				\label{wtequalsst}
				$(\cat{C},\cat{D})$ is a WT pair in $\cat{T}$ if and only if there is a finite silting collection $\silting$ in $\cat{T}$ such that
				\begin{itemize}
					\item $\Hom_\cat{T}(P,X)$ is finite-dimensional for all $X\in\cat{T}$ and $P\in\silting$,
					\item $\thick_\cat{T}(\silting)=\cat{C}$,
					\item $\cat{T}=\bigcup_{n\in\ints} \cat{T}^{t_\silting\leq n}$ and $\cat{D}=\bigcup_{n\in\ints} \cat{T}^{t_\silting\geq n}$.
				\end{itemize}
			\end{prop}
			\begin{proof}
				``$\imp$'': Let $\silting$ be a set of isomorphism representatives of the indecomposable objects in $\coheart_w$.
				By \cref{orthogonalitylemma}, $w$ is \textbf{w}-\textbf{t}-strictly left orthogonal to $t$ on $\cat{T}$.
				Thus
				\begin{equation*}
					\coheart_w=\cat{T}_{w\geq 0}\cap\cat{T}_{w\leq 0}=\prescript{\perp}{}{(\cat{T}^{t<0})}\cap\cat{T}^{t\leq 0}=\DProj_t(\cat{T}),
				\end{equation*}
				so $\silting$ consists of the indecomposable derived projectives.
				Moreover, by  \cite[Thm.~5.3.1]{BondarkoWeightTAndBack} (cf.~\cref{derprojadjacency} above) $\cat{T}$ has enough derived projectives with respect to $t$.
				It follows from \cref{siltingderprojconditions} that $\silting$ is a (by assumption finite) silting collection, and $t$ the associated silting \textbf{t}-structure.
				For $P\in\silting$ and $X\in\cat{T}$, \cref{derivedprojtruncationfullyfaithful} gives $\Hom_\cat{T}(P,X)\cong\Hom_{\heart_t}(H_t^0(P),H_t^0(X))$, which is finite-dimensional by assumption.
				
				Since $t$ is bounded above on $\cat{T}$, we have (by definition) $\cat{T}=\bigcup_{n\in\ints} \cat{T}^{t\leq n}$.
				Finally, if $X\in\cat{T}^{t\geq m}$ then $X\in\cat{T}^{t\geq m}\cap\cat{T}^{t\leq n}$ for some $n\in\ints$ as $t$ is bounded above, so $X\in\tria_\cat{T}(\heart_t)=\cat{D}$.
				The converse inclusion is obvious since $\bigcup_{n\in\ints} \cat{T}^{t\geq n}\subseteq\cat{T}$ is a triangulated subcategory and $\heart_t\subseteq\bigcup_{n\in\ints} \cat{T}^{t\geq n}$.
				
				``$\pmi$'': Note that $\heart_{t_\silting}$ is Hom-finite finite-length with finitely many simples by \cite[Prop.~4.6]{AdachiMizunoYang}.
				As $t_\silting$ is a silting \textbf{t}-structure on $\cat{T}$, it is non-degenerate.
				By \cref{siltingderprojconditions} $\cat{T}$ has enough derived projectives with respect to $t_\silting$ and $\silting$ is a full set of indecomposable derived projectives in $\cat{T}$.
				Therefore by \cite[Thm.~2.4]{CoelhoSimoesPauksztelloPloog} (see \cref{derprojadjacency} above) there is a weight structure $w$ on $\cat{T}$ that is left adjacent to $t_\silting$.
				As $t_\silting$ is bounded above, so is $w$.
				
				By \cref{orthogonalitylemma}, $w$ is \textbf{w}-\textbf{t}-strictly left orthogonal to $t_\silting$.
				As above we get $\coheart_w=\DProj_t(\cat{T})$, and therefore $\coheart_w=\Kar_\cat{T}(\silting)$ as $\silting$ consists of the indecomposable derived projectives.
				In particular $\coheart_w$ contains finitely many indecomposables up to isomorphism, and hence $\cat{C}=\thick_\cat{T}(\silting)=\thick_\cat{T}(\coheart_w)$.
				Finally, we have $\tria_\cat{T}(\heart_{t_\silting})\subseteq\tria_\cat{T}(\cat{T}^{t_\silting\geq 0})\subseteq\cat{D}$.
				Conversely, if $X\in\cat{D}$ then $X\in\cat{T}^{t_\silting\geq m}\cap\cat{T}^{t_\silting\leq n}$ for some $m,n\in\ints$ (as $t_\silting$ is bounded above), so $X\in\tria_\cat{T}(\heart_{t_\silting})$.
			\end{proof}
			In \cite{Fushimi} it is shown that for a WT pair $(\cat{C},\cat{D})$ in an algebraic triangulated category $\cat{T}$ there are bijections between bounded weight structures on $\cat{C}$ and bounded \textbf{t}-structures $\cat{D}$ with finite-length heart.
			By the examples listed in \cref{wtpairs}, this unifies several earlier results:
			for the WT pair $(\Kb(\rprojfg A),\Db(\rmodfd A))$ in $\Db(\rmodfd A)$, where $A$ is a finite-dimensional algebra, the theorem is due to Koenig and Yang \cite{KoenigYang}.
			It seems the version for non-positive dg algebras with finite-dimensional total cohomology stated in \cite{BruestleYang} was originally a folklore result.
			Recently \cite{SuYang} and \cite{Zhang} provided new proofs using Koszul duality for $A_\infty$-algebras and dg algebras, respectively.
			For homologically smooth non-positive dg algebras the theorem is stated in \cite{KellerNicolasCluster}.
			
			The following theorem is a slight refinement of the results in \cite[\S 5]{KoenigYang} and \cite[Thm.~4.8]{Fushimi}, making the bijections explicit.
			\begin{thm}[WT correspondence]\leavevmode
				\label{weighttbijections}
				Let $(\cat{C},\cat{D})$ be a WT pair in an algebraic triangulated category $\cat{T}$.
				Taking derived projective covers gives one of the eight bijections fitting in the following commutative diagram from \cite{KoenigYang} and \cite{Fushimi}:
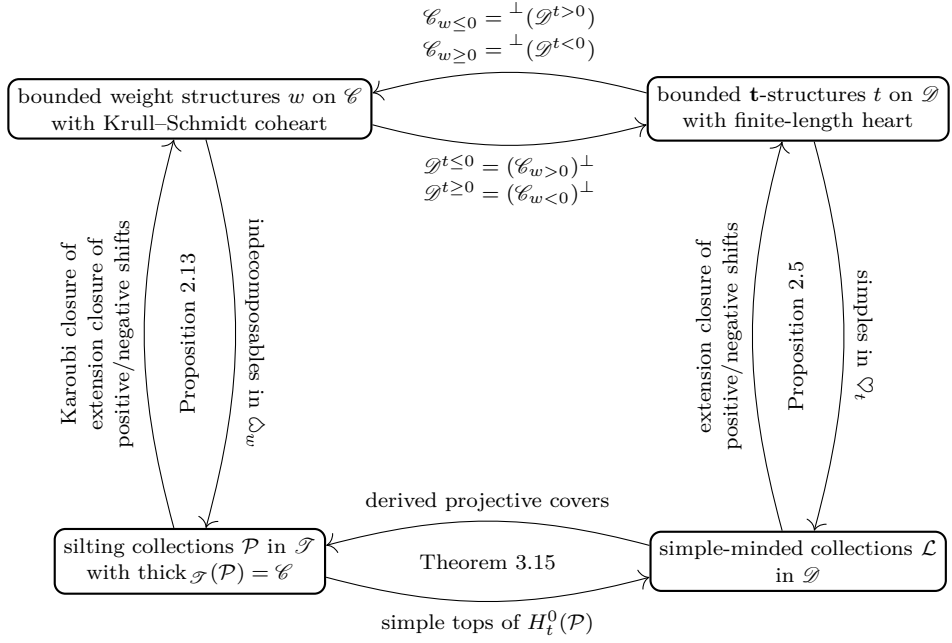
				\begin{equation*}
					\begin{tikzpicture}[bend angle=15,font=\normalfont\footnotesize]
						\node[align=center,draw,rounded corners,thick] (tstruct) at (4,3) {bounded \textbf{t}-structures $t$ on $\cat{D}$\\with finite-length heart};
						\node[align=center,draw,rounded corners,thick] (wstruct) at (-4,3) {bounded weight structures $w$ on $\cat{C}$\\with Krull--Schmidt coheart};
						\node[align=center,draw,rounded corners,thick] (simplem) at (4,-3) {simple-minded collections $\simplem$\\in $\cat{D}$};
						\node[align=center,draw,rounded corners,thick] (silting) at (-4,-3) {silting collections $\silting$ in $\cat{T}$\\with $\thick_\cat{T}(\silting)=\cat{C}$};
						
						\draw[bend right,-{>[scale=1.5]}] ([yshift=0.5\baselineskip] tstruct.west) to node[above,align=center,anchor=south] {$\cat{C}_{w\leq 0}=\prescript{\perp}{}{(\cat{D}^{t>0})}$\\$\cat{C}_{w\geq 0}=\prescript{\perp}{}{(\cat{D}^{t<0})}$} ([yshift=0.5\baselineskip] wstruct.east);
						\draw[bend right,-{>[scale=1.5]}] ([yshift=-0.5\baselineskip] wstruct.east) to node[below,align=center,anchor=north] {$\cat{D}^{t\leq 0}=(\cat{C}_{w>0})^\perp$\\$\cat{D}^{t\geq 0}=(\cat{C}_{w<0})^\perp$} ([yshift=-0.5\baselineskip] tstruct.west);
						
						\draw[bend left,-{>[scale=1.5]}] ([xshift=0.5\baselineskip] tstruct.south) to node[left,rotate=270,anchor=south] {simples in $\heart_t$} ([xshift=0.5\baselineskip] simplem.north) ;
						\draw[bend left,-{>[scale=1.5]}] ([xshift=-0.5\baselineskip] simplem.north) to node[right,align=center,rotate=90,anchor=south] {extension closure of\\positive\slash negative shifts} ([xshift=-0.5\baselineskip] tstruct.south);
						
						\draw[bend left,-{>[scale=1.5]}] ([xshift=0.5\baselineskip] wstruct.south) to node[left,rotate=270,anchor=south] {indecomposables in $\coheart_w$} ([xshift=0.5\baselineskip] silting.north);
						\draw[bend left,-{>[scale=1.5]}] ([xshift=-0.5\baselineskip] silting.north) to node[right,align=center,rotate=90,anchor=south] {Karoubi closure of\\extension closure of\\positive\slash negative shifts} ([xshift=-0.5\baselineskip] wstruct.south);
						
						\draw[bend right,-{>[scale=1.5]}] ([yshift=0.5\baselineskip] simplem.west) to node[above,align=center,anchor=south] {derived projective covers} ([yshift=0.5\baselineskip] silting.east);
						\draw[bend right,-{>[scale=1.5]}] ([yshift=-0.5\baselineskip] silting.east) to node[below,align=center,anchor=north] {simple tops of $H_t^0(\silting)$} ([yshift=-0.5\baselineskip] simplem.west);
						
						\node at ($(silting.east)!0.5!(simplem.west)$) {\cref{siltingderprojconditions}};
						\node at ($(silting.north)!0.5!(wstruct.south)$) {\rotatebox{90}{\cref{weighttosilting}}};
						\node at ($(simplem.north)!0.5!(tstruct.south)$) {\rotatebox{90}{\cref{ttosimplem}}};
					\end{tikzpicture}
				\end{equation*}
				The bijection at the top in particular says that the correspondence between weight structures and \textbf{t}-structures is given by \textbf{w}-\textbf{t}-strict orthogonality.
			\end{thm}
			\begin{proof}
				In view of the proof of \cite[Thm.~4.8]{Fushimi}, we essentially only have to show that the bijection between weight structures and \textbf{t}-structures is given by \textbf{w}-\textbf{t}-strict orthogonality.
				For convenience of the reader, we first briefly describe the other bijections.
				
				It is well-known (cf.~\cref{ttosimplem} above) that simple-minded collections in $\cat{D}$ are in bijection with bounded \textbf{t}-structures on $\cat{D}$ with finite-length heart.
				Similarly (cf.~\cref{weighttosilting,weighttoclassicalsilting} above), classical silting collections in $\cat{C}$ are in bijection with bounded weight structures on $\cat{C}$ with Krull--Schmidt coheart.
				By \cite[Prop.~5.2]{AdachiMizunoYang}, classical silting collections in $\cat{C}$ are the same as silting collections $\silting$ in $\cat{T}$ with $\thick_\cat{T}(\silting)=\cat{C}$.
				Thus the vertical maps are bijections.
				
				Let $\silting$ be a silting collection in $\cat{T}$ with $\thick_\cat{T}(\silting)=\cat{C}$, and $t=t_\silting$ the \textbf{t}-structure on $\cat{T}$ defined by $\silting$.
				By \cref{siltingderprojconditions}, $\silting$ consists of the derived projective covers of the simple objects in $\heart_t$, which form a simple-minded collection $\simplem$ in $\cat{D}$ by \cite[Prop.~5.2 and Prop.~4.6]{AdachiMizunoYang}.
				From the definition of derived projective covers it is clear that $\simplem$ consists of the simple tops of $H_t^0(\silting)$.
				Sending $\silting$ to $\simplem$ defines a bijection by \cite[Prop.~4.6]{Fushimi}.
				
				It remains to show that the induced bijection between weight structures and \textbf{t}-structures is given by \textbf{w}-\textbf{t}-strict orthogonality.
				Let $w$ be a bounded weight structure on $\cat{C}$.
				By the above, $w$ is a bounded weight structures obtained from a silting collection $\silting$, and the corresponding \textbf{t}-structure on $\cat{D}$ is the restriction of the associated silting \textbf{t}-structure $t=t_\silting$ on $\cat{T}$.
				Hence \textbf{t}-strict left orthogonality is obvious.
				
				For \textbf{w}-strictness, let $X\in\cat{C}\cap\prescript{\perp}{}{(\cat{D}^{t>0})}$ and let $t_{\leq 0}X\to X\to t_{>0}X\to t_{\leq 0}X[1]$ be the \textbf{t}-decomposition of $X$ with respect to the \textbf{t}-structure $t$ on $\cat{T}$.
				As $t$ is bounded above on $\cat{T}$, we have $t_{>0}X\in\tria_\cat{T}(\heart_t)=\cat{D}$, so $t_{>0}X\in\cat{D}^{t>0}$ and thus $\Hom_\cat{T}(X,t_{>0}X)=0$ by assumption.
				Hence the \textbf{t}-decomposition triangle splits and gives $t_{\leq 0}X\cong t_{>0}X[-1]\oplus X$, which implies $X\cong t_{\leq 0}X\in\cat{T}^{t\leq 0}\cap\cat{C}=\cat{T}_{w\leq 0}\cap\cat{C}=\cat{C}_{w\leq 0}$.
				
				Let $X\in\cat{C}\cap\prescript{\perp}{}{(\cat{D}^{t<0})}$.
				As $w$ is bounded on $\cat{C}$, there is $n\in\ints$ with $X\in\cat{C}_{w>n}$.
				Let $Y\in\cat{T}^{t<0}$ and consider the \textbf{t}-decomposition $t_{\leq n}Y\to Y\to t_{>n}Y\to t_{\leq n}Y[1]$.
				Since $w$ is left adjacent to $t$ on $w$, we have $t_{\leq n}Y\in\cat{T}_{w\leq n}$, and since $t$ is bounded above on $\cat{T}$, we have $t_{>n}Y\in \cat{D}^{t<0}$.
				Since $X\in\cat{C}_{w>n}$ and $X\in\prescript{\perp}{}{(\cat{D}^{t<0})}$, applying $\Hom_\cat{T}(X,-)$ to the \textbf{t}-decomposition of $Y$ shows that $\Hom_\cat{T}(X,Y)=0$.
				Hence from \cref{adjacencylemma} we get $X\in\prescript{\perp}{}{(\cat{T}^{t<0})}\cap\cat{C}=\cat{T}_{w\geq 0}\cap\cat{C}=\cat{C}_{w\geq 0}$.
			\end{proof}
	\section{Koszul duality between simple-minded and silting collections}\label{koszulduality}
		In this section we prove a theorem about Koszul duality between finite simple-minded collections and finite silting collections.
		\subsection{dg Koszul duality}
			Recall that a \emph{dg enhancement} (originally called \emph{enhancement} in \cite{BondalKapranov}) of a triangulated category $\cat{T}$ is a \emph{pretriangulated dg category} $\widetilde{\cat{T}}$ in the sense of \cite[\S 3, Def.~1]{BondalKapranov} together with an equivalence $\cat{T}\cong H^0(\widetilde{\cat{T}})$.
			An object $X\in\cat{T}$ yields a dg functor $\Hom_{\widetilde{\cat{T}}}(X,-)\: \widetilde{\cat{T}}\to\dgrMod\End_{\widetilde{\cat{T}}}(X)$, which induces a triangulated functor $\Hom_{\widetilde{\cat{T}}}(X,-)\: \cat{T}\to\D(\End_{\widetilde{\cat{T}}}(X))$.
			
			Note that by \cite[\S 7.5]{KrauseChicago} and the first part of the proof of \cite[Thm.~4.3]{KellerDerivingDGCats}, dg-enhanced triangulated categories are precisely the stable categories of Frobenius categories considered in \cite{KellerDerivingDGCats}.
			\begin{ex}
				A dg enhancement $\widetilde{\D}(\cat{A})$ of the derived category $\D(\cat{A})$ of a dg category is given by the dg category of K-projective dg $\cat{A}$-modules, see \cite[\S 4.1]{KellerDerivingDGCats} for details.
				Here a dg $\cat{A}$-module $P$ is \emph{K-projective} if $H^0(\Hom_{\dgrMod\cat{A}}(P,N))=0$ for all acyclic dg $\cat{A}$-modules $N$.
				A dg $\cat{A}$-module $M$ can be viewed as an object of $\widetilde{\D}(\cat{A})$ by replacing it by a K-projective resolution, i.e.~a K-projective dg module $pM$ that is quasi-isomorphic to $M$.
				Note that by \cite[\S 3.1, p.~70]{KellerDerivingDGCats} K-projectivity is equivalent to the \emph{property (P)} considered in \cite{KellerDerivingDGCats}, and in particular K-projective resolutions are precisely the P-resolutions defined in \cite[\S 3.1]{KellerDerivingDGCats}.
				
				We abbreviate $\Hom_{\widetilde{\D}(\cat{A})}(-,-)=\RHom_\cat{A}(-,-)$, and by slight abuse of notation we also write $\REnd_\cat{A}(X)=\RHom_\cat{A}(X,X)$ for $X\in\D(\cat{A})$.
				Dually one can also use K-injective resolutions instead.
			\end{ex}
			The definition of the Koszul dual for non-positive dg algebras and positive dg algebras is a special case of the following general definition of Koszul duality for augmented dg categories from \cite[\S 10.2]{KellerDerivingDGCats}.
			\begin{dfn}\leavevmode
				\label{dgkoszuldual}
				\begin{enumerate}
				\item An \emph{augmented dg category} is a dg category $\cat{A}$ with pairwise non-isomorphic objects such that for every object $A\in\cat{A}$ there is a dg $\cat{A}$-module $\ol{A}$ (called \emph{augmenting dg module}) such that
					\begin{equation*}
						H^k(\ol{A})(B)\cong\begin{cases}
							\groundfield&\text{if }k=0, A=B,\\
							0&\text{otherwise}.
						\end{cases}
					\end{equation*}
				\item The \emph{dg Koszul dual} of an augmented dg category $\cat{A}$ is the dg category $\cat{A}^{!,\dg}$ with objects $\{A^!\mid A\in \cat{A}\}$ and morphisms
					\begin{equation*}
						\cat{A}^{!,\dg}(A^!,A'^!)=\RHom_\cat{A}(\ol{A},\ol{A'})=\Hom_{\dgrMod\cat{A}}(p\ol{A},p\ol{A'}),
					\end{equation*}
					where $p\ol{A}$ and $p\ol{A'}$ are K-projective resolutions of $\ol{A}$ and $\ol{A'}$.
				\end{enumerate}
			\end{dfn}
			\begin{rem}\leavevmode
				\begin{enumerate}
					\item $\cat{A}^{!,\dg}$ is well-defined only up to quasi-equivalence, cf.~\cite[\S 10.2]{KellerDerivingDGCats}.
					\item In \cite{KellerDerivingDGCats} the Koszul dual is defined more abstractly as a lift of the augmenting modules $\{\ol{A}\mid A\in\cat{A}\}$.
						With the notation from \cref{dgkoszuldual} the dg category $\cat{A}^{!,\dg}$ and the $\cat{A}^{!,\dg}$-$\cat{A}$-bimodule $\bigoplus_{B\in\cat{A}} p\ol{B}$ provide a lift, and hence the ``abstract'' definition agrees with the ``concrete'' definition we use here.
					\item By \cite[\S 10.2]{KellerDerivingDGCats}, $\cat{A}^{!,\dg}$ becomes an augmented dg category with augmenting dg modules $\ol{A^!}=\RHom_\cat{A}(\bigoplus_{B\in\cat{A}} p\ol{B},D\cat{A}(A,-))$.
						Here $D=\Hom_{\dgrMod\groundfield}(-,\groundfield)$ denotes the $\groundfield$-linear duality functor.
					\item If the augmenting dg modules $\ol{A}$ are compact and generate $\D(\cat{A})$ as triangulated subcategory closed under arbitrary coproducts (or, equivalently, if $\thick\{\ol{A}\mid A\in\cat{A}\}=\D^\comp(\cat{A})$), then by \cite[Lemma~10.5 ``The finite case'']{KellerDerivingDGCats} the $\cat{A}^{!,\dg}$-$\cat{A}$-dg bimodule $\bigoplus_{B\in\cat{A}} p\ol{B}$ provides an equivalence of categories $\D(\cat{A}^{!,\dg})\to\D(\cat{A})$, sending $\cat{A}^{!,\dg}(-,A^!)$ to $\ol{A}$ for all $A\in\cat{A}$.
				\end{enumerate}
			\end{rem}
			Important special cases (for us) are the dg Koszul duals of non-positive and positive dg algebras.
			For a more detailed discussion of Koszul duality in these cases see also \cite{FushimiDgKoszul}.
			\begin{ex}\leavevmode
				\label{negativepositivekoszulduality}
				Let $A$ be a dg algebra such that $H^0(A)$ is finite-dimensional and all simple $H^0(A)$-modules are $1$-dimensional.
				Take a complete set of primitive orthogonal idempotents of $H^0(A)$ and suppose they lift to idempotents $\{e_i\mid i\in I\}$ in $A$ with $\sum_{i\in I} e_i=1$.
				Consider $A=\bigoplus_{i,j\in I} e_iAe_j$ as a dg category with objects $I$ and morphisms $A(j,i)=e_iAe_j$.
				\begin{enumerate}
					\item If $A$ is a non-positive dg algebra, we can consider the simple $H^0(A)$-modules as dg $A$-modules concentrated in degree $0$ via the quasi-isomorphism $t_{\leq 0}A\to A$ and the quotient map $t_{\leq 0}A\tto H^0(A)$ where $t_{\leq 0}A$ denotes the truncation of $A$ to degrees~$\leq 0$.
						These ``simple dg $A$-modules'' make $A$ an augmented dg category, and (viewed as a dg algebra by taking the direct sum over the finitely many objects) the dg Koszul dual of $A$ is $A^{!,\dg}=\REnd_A(L)$ where $L$ is the sum of the simple $H^0(A)$-modules viewed as dg $A$-modules.
						Note that $A^{!,\dg}$ is a positive dg algebra as a consequence of \cite[Thm.~A.1]{BruestleYang}.
					\item If $A$ is a positive dg algebra, by \cite[Cor.~4.7]{KellerNicolas} there are unique dg $A$-modules $L_i$ corresponding to the simple $H^0(A)$-modules $e_iH^0(A)$, and these make $A$ an augmented dg category.
						Then the dg Koszul dual of $A$ is $A^{!,\dg}=\REnd_A(\bigoplus_i L_i)$, cf.~\cite[Rem.~5.1]{KellerNicolas}.
						This is a non-positive dg algebra as a consequence of \cite[Lemma~5.2]{KellerNicolas}.
				\end{enumerate}
			\end{ex}
		\subsection{Koszul duality between simple-minded and silting collections}
			The following lemma provides a convenient description of certain subcategories of a compactly generated dg-enhanced triangulated category with a compact silting collection.
			This is a slight generalization of \cite[Lemma~3.1]{KalckYangII}, although its proof uses essentially the same arguments.
			\begin{lemma}
				\label{siltingequivalence}
				Let $\cat{T}=H^0(\widetilde{\cat{T}})$ be a compactly generated dg-enhanced triangulated category.
				For a compact silting collection $\silting$ in $\cat{T}$ such that $\End_\cat{T}(\bigoplus_{P\in\silting} P)$ is finite-dimensional, let $t$ be its associated silting \textbf{t}-structure and $\cat{D}\subseteq\cat{T}$ be the triangulated subcategory generated by the simple objects in $\heart_t$.
				Let $E=\End_{\widetilde{\cat{T}}}(\bigoplus_{P\in\silting}P)$.
				\begin{enumerate}
					\item There is an equivalence $\D(E)\to\cat{T}$ that takes the simple $E$-modules to the simple objects in $\heart_t$.
						Moreover it identifies $\perf(E)$ with $\thick_\cat{T}(\silting)$, and $\Dfd(E)$ with $\cat{D}$.
					\item  If $H^*(E)$ is finite-dimensional, then $\thick_\cat{T}(\silting)=\cat{T}^\comp\subseteq \cat{D}$.
				\end{enumerate}
			\end{lemma}
			\begin{proof}\leavevmode
				\begin{enumerate}
					\item As $\silting$ weakly generates $\cat{T}$ by \cref{siltingtnondeg}, $\Hom_{\widetilde{\cat{T}}}(\bigoplus_{P\in\silting} P,-)\: \cat{T}\to\D(E)$ is an equivalence by (the proof of) \cite[Thm.~4.3]{KellerDerivingDGCats}.
						Its inverse takes $E$ to $\bigoplus_{P\in\silting}P$, and hence identifies $t=(\silting^{\perp_{>0}},\silting^{\perp_{<0}})$ with the standard \textbf{t}-structure on $\D(E)$, as this is the silting \textbf{t}-structure associated with the silting object $E$ in $\D(E)$.
						In particular it also identifies the simple objects in the hearts.
						The rest is clear since (by definition) $\cat{D}$ is the triangulated subcategory generated by the simple objects of $\heart_t$, while on the other side $\Dfd(E)$ is the triangulated subcategory generated by the simple dg $E$-modules (note that these lie in $\Dfd(E)$, since $H^0(E)=\End_\cat{T}(\bigoplus_{P\in\silting} P)$ is finite-dimensional).
					\item The assumption that $H^*(E)$ is finite-dimensional ensures that $\perf(E)\subseteq\Dfd(E)$, and thus we get $\thick_\cat{T}(\silting)\subseteq\cat{D}$ from 1).\qedhere
				\end{enumerate}
			\end{proof}
			The following result establishes a Koszul duality between simple-minded collections and silting collections, which was suggested by Keller.
			\begin{thm}
				\label{koszuldualitygeneral}
				Let $\cat{T}=H^0(\widetilde{\cat{T}})$ be a compactly generated dg-enhanced triangulated category.
				For a compact silting collection $\silting$ in $\cat{T}$ such that $\End_\cat{T}(\bigoplus_{P\in\silting} P)$ is finite-dimensional, let $\simplem$ be the set of simple objects in the heart of the silting \textbf{t}-structure associated with $\silting$, and suppose that $\End_\cat{T}(L)$ is $1$-dimensional for each $L\in\simplem$.
				\begin{enumerate}
					\item The dg algebra $\End_{\widetilde{\cat{T}}}(\bigoplus_{L\in\simplem}L)$ is the dg Koszul dual of $\End_{\widetilde{\cat{T}}}(\bigoplus_{P\in\silting}P)$.
					\item If $H^n(\End_{\widetilde{\cat{T}}}(\bigoplus_{P\in\silting}P))$ is finite-dimensional for all $n\in\ints$, then $\End_{\widetilde{\cat{T}}}(\bigoplus_{P\in\silting}P)$ is the dg Koszul dual of $\End_{\widetilde{\cat{T}}}(\bigoplus_{L\in\simplem}L)$.
				\end{enumerate}
			\end{thm}
			\begin{proof}
				For brevity we write $P=\bigoplus_{P'\in\silting} P'$ and $L=\bigoplus_{L'\in\simplem} L'$.
				\begin{enumerate}
					\item By definition, the cohomology of $E=\End_{\widetilde{\cat{T}}}(P)$ is given by
						\begin{equation*}
							H^n(\End_{\widetilde{\cat{T}}}(P))=\Hom_\cat{T}(P,P[n])
						\end{equation*}
						and therefore is concentrated in non-positive degrees.
						By \cref{negativepositivekoszulduality} the Koszul dual of the non-positive dg algebra $E=\End_{\widetilde{\cat{T}}}(P)$ (note that the identity morphisms provide the required lifts of idempotents) is given by
						\begin{equation*}
							E^{!,\dg}=\REnd_E(L_E),
						\end{equation*}
						where $L_E$ is the sum of the simple $H^0(E)$-modules viewed as dg $E$-modules concentrated in degree $0$.
						To compute this, we use the equivalence $\D(E)\to\cat{T}$ from \cref{siltingequivalence}, which takes $L_E$ to $L$ and therefore provides a quasi-isomorphism
						\begin{equation*}
							\End_{\widetilde{\cat{T}}}(P)^{!,\dg}=E^{!,\dg}=\REnd_E(L_E)\simeq\End_{\widetilde{\cat{T}}}(L).
						\end{equation*}
					\item This follows from \cite[Thm.~4.17]{FushimiDgKoszul}, since $\End_{\widetilde{\cat{T}}}(L)$ is the dg Koszul dual of $\End_{\widetilde{\cat{T}}}(P)$ by 1).\qedhere
				\end{enumerate}
			\end{proof}
			In the case of finite-dimensional algebras (and analogously for non-positive dg algebras with finite-dimensional total cohomology) one can prove \cref{koszuldualitygeneral}.2) more directly.
			The proof is interesting since it uses an approach that was used in \cite{Zhang} to construct silting collections corresponding to simple-minded collections in $\Dfd(A)$, where $A$ is a non-positive dg algebra with finite-dimensional total cohomology.
			\begin{thm}
				Let $A$ be a finite-dimensional algebra.
				Let $\simplem$ be a simple-minded collection in $\Db(\rmodfd A)$ such that $\End_{\Db(\rmodfd A)}(L)$ is $1$-dimensional for each $L\in\simplem$, and $\silting$ be the corresponding classical silting collection in $\Kb(\rprojfg A)$ under the bijection from \cref{weighttbijections}.
				Then $\REnd_A(\bigoplus_{P\in\silting} P)$ is the dg Koszul dual of $\REnd_A(\bigoplus_{L\in\simplem} L)$.
			\end{thm}
			\begin{proof}
				For brevity we write $L=\bigoplus_{L'\in\simplem} L'$.
				The cohomology of $E^!=\REnd_A(L)$ is given by
				\begin{equation*}
					H^n(E^!)=H^n(\REnd_A(L))=\Hom_{\Db(\rmodfd A)}(L,L[n])
				\end{equation*}
				and therefore is concentrated in non-negative degrees, and moreover
				\begin{equation*}
					H^0(E^!)=\Hom_{\Db(\rmodfd A)}(L,L)=\bigoplus_{L'\in\simplem} \End_{\Db(\rmodfd A)}(L')
				\end{equation*}
				is semisimple.
				Hence by definition the Koszul dual of $E^!$ is
				\begin{equation*}
					E=\REnd_{E^!}(H^0(E^!)).
				\end{equation*}
				Let $L^0$ be the sum of the simple $A$-modules and $A^!=\REnd_A(L^0)$ the Koszul dual of $A$ viewed as a dg algebra concentrated in non-positive degrees.
				As is explained in \cite{Zhang} we obtain a commutative diagram
				\begin{equation}
					\label{koszuldualitydiagram}
					\begin{tikzcd}
						\perf(E^!)\arrow{r}{\Psi}
						&\perf(A^!)\arrow{r}{\Phi}
						&\Db(\rmodfd A)\\
						\Dfd(E^!)\arrow{r}{\Psi}\arrow[hook]{u}
						&\Dfd(A^!)\arrow{r}{\Phi}\arrow[hook]{u}
						&\Kb(\rinjfg A)\arrow{r}{\nu^{-1}}\arrow[hook]{u}
						&\Kb(\rprojfg A)
					\end{tikzcd}
				\end{equation}
				where $\nu^{-1}$ is the inverse Nakayama functor and the horizontal functors $\Psi$ and $\Phi$ are equivalences defined by
				\begin{align*}
					\Phi&=-\Lotimes_{A^!} L^0,&
					\Psi&=-\Lotimes_E \Phi^{-1}(L).
				\end{align*}
				Note that the definition of $\Psi$ implicitly also uses the equivalence induced by the quasi-isomorphism $E^!\cong\REnd_{A^!}(\Phi^{-1}(L))$ induced by $\Phi$, which we leave out for brevity.
				
				Now observe that by construction of the diagram \cref{koszuldualitydiagram} the equivalences in the bottom row map $H^0(\REnd_A(L))=\End_{\Db(\rmodfd A)}(L)$ to $P$, and therefore we obtain a quasi-isomorphism
				\begin{equation*}
					\REnd_A(L)^!=E\cong\REnd_A(P).\qedhere
				\end{equation*}
			\end{proof}
			\begin{rem}
				Koszul duality of $\End_{\widetilde{\cat{T}}}(\bigoplus_{P\in\silting} P)$ and $\End_{\widetilde{\cat{T}}}(\bigoplus_{L\in\simplem} L)$ does not imply that $\silting$ and $\simplem$ correspond to each other.
				For a (trivial) counterexample one can simply shift $\simplem$ or $\silting$, and for further non-trivial examples with the same dg algebras occuring see \cref{koszulstd,koszulnonstdfaithful}.
			\end{rem}
		\subsection{Some small examples: the \texorpdfstring{$A_2$}{A2} quiver}
			We illustrate \cref{koszuldualitygeneral} by some examples over the algebra $A=\groundfield(2\to 1)$.
			For a simple-minded collection $\simplem$ in $\Db(\rmodfd A)$ and a silting collection $\silting$ in $\Kb(\rprojfg A)$, let $E=\REnd_A(\bigoplus_{P\in\silting} P)$ and $E^!=\REnd_A(\bigoplus_{L\in\simplem}L)$.
			To compute the dg Koszul duals of $E$ and $E^!$ we use the description of the dg Koszul dual from \cref{negativepositivekoszulduality}.
			
			Recall that for a dg algebra $B$ and dg $B$-modules $X$ and $Y$, the dg algebra $\REnd_B(X,Y)$ can be computed by replacing both $X$ and $Y$ by K-projective resolutions, i.e.~quasi-isomorphic perfect dg modules.
			Replacing both is convenient to determine the composition of morphisms, as otherwise one would have to use formal inverses to quasi-isomorphisms.
			The degree $n$ part of $\REnd_B(X,Y)$ consists of all $B$-linear morphisms $X\to Y[n]$ (not necessarily dg morphisms), and the differential is defined by $d(f)=df-(-1)^{|f|}fd$.
			Alternatively one can use K-injective resolutions.
			If $B$ has trivial differential, K-projective resolutions are just projective resolutions.
			\begin{ex}[The standard example]
				\label{koszulstd}
				Consider the standard simple-minded collection $\simplem=\{1,2\}$ consisting of the simple $A$-modules, and the corresponding standard silting collection $\silting=\{\begin{smallmatrix}1\\2\end{smallmatrix},2\}$ consisting of the indecomposable projective $A$-modules.
				\begin{enumerate}
					\item We have $E=\REnd_A(A)\cong A$, viewed as a non-positive dg algebra concentrated in degree~$0$ with trivial differential.
					\item To compute $E^!=\REnd_A(1\oplus 2)$ we replace the non-projective simple by its projective resolution: $1\cong (2\to\begin{smallmatrix}1\\2\end{smallmatrix})$.
						From this it follows that $E^!$ is $7$-dimensional, as it is the direct sum of
						\begin{align*}
							\RHom_A(2,2)&=\groundfield e_2,&
							\RHom_A(2,1)&=\groundfield f_0\oplus\groundfield f_{-1},\\
							\RHom_A(1,2)&=\groundfield g,&
							\RHom_A(1,1)&=\groundfield e_{11}\oplus\groundfield e_{12}\oplus\groundfield h,
						\end{align*}
						with the degrees and the differentials of the basis elements given by
						\begin{equation*}
							\begin{array}{c||c|c|c|c|c|c|c}
								&e_{11}&e_{12}&e_2&f_0&f_{-1}&g&h\\
								\hline
								\deg&0&0&0&0&-1&1&1\\
								d&h&-h&0&0&f_0&0&0
							\end{array}
						\end{equation*}
						The morphisms $e_{11}$, $e_{12}$ and $e_2$ are orthogonal idempotents, and the algebra structure of $E^!$ is given by the quiver with relations
						\begin{equation*}
							E^!=\groundfield\left(
							\begin{tikzcd}[column sep=small]
								e_{11}\arrow{rr}{h}\arrow[bend left]{rd}{g}&&e_{12}\\
								&e_2\arrow{ru}[swap]{f_0}\arrow[bend left]{lu}{f_{-1}}&
							\end{tikzcd}
							\right)/\begin{smallpmatrix}f_0g=h\\f_{-1}g=e_{11}\\hf_{-1}=f_0\\gf_{-1}=e_2\end{smallpmatrix}.
						\end{equation*}
						It follows that the cohomology $H^*(E^!)=\Ext_A^*(1\oplus 2,1\oplus 2)$ has a basis consisting of the classes of $e_1=e_{11}+e_{12}$, $e_2$ and $g$, where $g$ spans the $1$-dimensional $\Ext_A^1(1,2)$.
						It is easy to see that the map $H^*(E^!)\to E^!$ defined by sending this basis of $H^*(E^!)$ to these representatives is a quasi-isomorphism.
					\item As $E\cong A$ as dg algebras, there is nothing to do: the dg Koszul dual of $E$ is literally $\REnd_A(1\oplus 2)=E^!$.
					\item To compute the dg Koszul dual of $E^!$ we use the quasi-isomorphism  $E^!\cong H^*(E^!)=\groundfield(e_1\longto{g}e_2)$ with $|g|=1$ and trivial differential.
						As the differential is trivial and $E^!$ is actually (not just cohomologically) concentrated in positive degrees, the simple dg $E^!$-modules are just the simple modules over $H^0(E^!)=(E^!)^0$ with trivial action of $(E^!)^{>0}$.
						K-projective resolutions of these are given by
						\begin{align*}
							e_1&\cong \left(\begin{tikzcd}[row sep=small, column sep=small,ampersand replacement=\&]
								0\\\groundfield
							\end{tikzcd}\right)
							,&
							e_2&\cong\left(\begin{tikzcd}[row sep=small, column sep=small,ampersand replacement=\&]
								\groundfield\arrow{rd}{g}\&0\\
								\groundfield\arrow{r}[swap]{d}\&\groundfield
							\end{tikzcd}\right)
						\end{align*}
						Here the top row indicates the vertex $e_2$ and the bottom row the vertex $e_1$, and in both cases the left-most non-zero term is in degree~$0$.
						From this it follows that the dg Koszul dual is
						\begin{equation*}
							\REnd_{E^!}(e_1\oplus e_2)=\groundfield\left(
							\begin{tikzcd}[column sep=small]
								E_{21}\arrow{rr}{H}\arrow[bend left]{rd}{G}
								&&E_{22}\\
								&E_1\arrow{ru}[swap]{F_1}\arrow[bend left]{lu}{F_0}&
							\end{tikzcd}
							\right)/\begin{smallpmatrix}F_1G=H\\F_0G=E_{21}\\HF_0=F_1\\GF_0=E_1\end{smallpmatrix}
						\end{equation*}
						with dg structure
						\begin{equation*}
							\begin{array}{c||c|c|c|c|c|c|c}
								&E_{21}&E_{22}&E_1&F_0&F_1&G&H\\
								\hline
								\deg&0&0&0&0&1&0&1\\
								d&-H&H&0&F_1&0&0&0
							\end{array}
						\end{equation*}
						By similar arguments as in 2), this dg algebra is quasi-isomorphic to its cohomology, which is $\groundfield(E_2\longto{G}E_1)\cong E$ with $E_2=E_{21}+E_{22}$.
				\end{enumerate}
			\end{ex}
			\begin{ex}[Non-standard, faithful heart]
				\label{koszulnonstdfaithful}
				Consider the simple-minded collection $\simplem=\{\begin{smallmatrix}1\\2\end{smallmatrix},2[1]\}$ and the corresponding silting collection $\silting=\{\begin{smallmatrix}1\\2\end{smallmatrix},1\}$.
				This is obtained from the standard example by left mutation at $2$.
				\begin{enumerate}
					\item We have $E=\REnd_A(\begin{smallmatrix}1\\2\end{smallmatrix}\oplus 1)\cong\groundfield(*_{\begin{smallmatrix}1\\2\end{smallmatrix}}\longto{x} *_1)\cong A$, with $|x|=0$ and trivial differential.
						Explicitly, $x$ is the morphism $x\: \begin{smallmatrix}1\\2\end{smallmatrix}\to 1$.
					\item We have $E^!=\REnd_A(\begin{smallmatrix}1\\2\end{smallmatrix}\oplus 2[1])\cong\groundfield(*_{2[1]}\longto{y}*_{\begin{smallmatrix}1\\2\end{smallmatrix}})$ with $|y|=1$ and trivial differential, where $y\: 2[1]\to \begin{smallmatrix}1\\2\end{smallmatrix}[1]$.
					\item Note that $E\cong A$ and therefore the dg Koszul dual of $E$ is $\REnd_A(1\oplus 2)$ as described in \cref{koszulstd} above.
						A quasi-isomorphism $E^!\to\REnd_A(1\oplus 2)$ is given by $*_{\begin{smallmatrix}1\\2\end{smallmatrix}}\mapsto e_2$, $*_{2[1]}\mapsto e_{11}+e_{12}$, $y\mapsto g$, as mentioned in \cref{koszulstd}.
					\item We already computed the Koszul dual of $E^!$ in \cref{koszulstd}, where we saw that it is quasi-isomorphic to $E$.
				\end{enumerate}
			\end{ex}
			\begin{ex}[Non-standard, non-faithful heart]
				\label{koszulnonstdnonfaithful}
				Consider the simple-minded collection $\simplem=\{1,2[-1]\}$ and the corresponding silting collection $\silting=\{\begin{smallmatrix}1\\2\end{smallmatrix},2[-1]\}$.
				This is obtained from the standard example by right mutation at $2$.
				The corresponding heart is semisimple.
				\begin{enumerate}
					\item We have $E=\REnd_A(\begin{smallmatrix}1\\2\end{smallmatrix}\oplus 2[-1])\cong\groundfield(*_{2[-1]}\longto{x}*_1)$ with $|x|=-1$ and trivial differential.
						Explicitly, $x$ is the morphism $2[-1]\to \begin{smallmatrix}1\\2\end{smallmatrix}[-1]$.
					\item As $\simplem$ is obtained from the standard simple-minded collection by shifting one object, the algebra structure of $E^!=\REnd_A(1\oplus 2[-1])$ is the same as in \cref{koszulstd}.
						However the degrees and differentials are now given by
						\begin{equation*}
							\begin{array}{c||c|c|c|c|c|c|c}
								&e_{11}&e_{12}&e_2&f_0&f_{-1}&g&h\\
								\hline
								\deg&0&0&0&-1&-2&2&1\\
								d&h&-h&0&0&-f_0&0&0
							\end{array}
						\end{equation*}
					\item To compute the dg Koszul dual of $E$, we need to take K-projective resolutions of the two simple dg $E$-modules $*_{2[-1]}$ and $*_1$.
						These are given by
						\begin{align*}
							*_{2[-1]}&\cong \left(\begin{tikzcd}[row sep=small, column sep=small,ampersand replacement=\&]
								0\\\groundfield
							\end{tikzcd}\right)
							,&
							*_1&\cong\left(\begin{tikzcd}[row sep=small, column sep=small,ampersand replacement=\&]
								0\&0\&\groundfield\arrow{ld}{x}\\
								\groundfield\arrow{r}[swap]{d}\&\groundfield\&0
							\end{tikzcd}\right)
						\end{align*}
						with the right-most terms in degree $0$.
						Here the top row represents the vertex $*_1$ and the bottom row the vertex $*_{2[-1]}$.
						It follows that
						\begin{equation*}
							\REnd_E(*_{2[-1]}\oplus *_1)\cong\groundfield\left(
							\begin{tikzcd}[column sep=small]
								E_{11}\arrow{rr}{H}\arrow[bend left]{rd}{G}&&E_{12}\\
								&E_2\arrow{ru}[swap]{F_0}\arrow[bend left]{lu}{F_{-1}}&
							\end{tikzcd}
							\right)/\begin{smallpmatrix}F_0G=H\\F_{-1}G=E_{11}\\HF_{-1}=F_0\\GF_{-1}=E_2\end{smallpmatrix}.
						\end{equation*}
						with the degrees and differentials given by
						\begin{equation*}
							\begin{array}{c||c|c|c|c|c|c|c}
								&E_{11}&E_{12}&E_2&F_0&F_{-1}&G&H\\
								\hline
								\deg&0&0&0&-1&-2&2&1\\
								d&H&-H&0&0&-F_0&0&0
							\end{array}
						\end{equation*}
						It is obvious that $\REnd_E(*_1\oplus *_2)$ is (quasi-)isomorphic to $E^!$.
					\item Similarly to the previous examples, it follows that $E^!$ is quasi-isomorphic to its cohomology, which is $H^*(E^!)=\groundfield(e_1\longto{g}e_2)$ where $e_1=e_{11}+e_{12}$ and $|g|=2$.
						The dg Koszul dual of $E^!$ is computed similarly to \cref{koszulstd}:
						the K-projective resolutions of the simple dg $E^!$-modules are given by
						\begin{align*}
							e_1&\cong \left(\begin{tikzcd}[row sep=small, column sep=small,ampersand replacement=\&]
								0\\\groundfield
							\end{tikzcd}\right)
							,&
							e_2&\cong\left(\begin{tikzcd}[row sep=small, column sep=small,ampersand replacement=\&]
								\groundfield\arrow{rrd}{y}\&0\&0\\
								0\&\groundfield\arrow{r}[swap]{d}\&\groundfield
							\end{tikzcd}\right),
						\end{align*}
						and from this we get
						\begin{equation*}
							\REnd_{E^!}(e_1\oplus e_2)=\groundfield\left(
							\begin{tikzcd}[column sep=small,ampersand replacement=\&]
								E_{21}\arrow{rr}{H}\arrow[bend left]{rd}{G}\&\&E_{22}\\
								\&E_1\arrow{ru}[swap]{F_2}\arrow[bend left]{lu}{F_1}\&
							\end{tikzcd}
							\right)/\begin{smallpmatrix}F_2G=H\\F_1G=E_{21}\\HF_1=F_2\\GF_1=E_1\end{smallpmatrix}
						\end{equation*}
						with dg structure
						\begin{equation*}
							\begin{array}{c||c|c|c|c|c|c|c}
								&E_{21}&E_{22}&E_1&F_1&F_2&G&H\\
								\hline
								\deg&0&0&0&1&2&-1&1\\
								d&-H&H&0&F_2&0&0&0
							\end{array}
						\end{equation*}
						Similarly to the previous examples, $\REnd_{E^!}(e_1\oplus e_2)$ is quasi-isomorphic to its cohomology, which is $\groundfield(E_2\longto{G}E_1)=E$ where $E_2=E_{21}+E_{22}$.
				\end{enumerate}
			\end{ex}
	\section{Naturality of orthogonality}\label{naturality}
		In this section we show that \textbf{w}-\textbf{t}-strict orthogonality is natural with respect to weight exact functors and \textbf{t}-exact functors.
		This is a slight generalization of the results from \cite[Prop.~4.4.5]{Bondarko} for adjacent \textbf{t}-structures.
		The main result in this section (\cref{orthogonalitynaturality}) is proved in essentially the same way except for the more technical notation required to set up the statements, which simplifies a lot in most interesting cases.
		\begin{thm}
			\label{orthogonalitynaturality}
			Let $\cat{C}$, $\cat{C}'$, $\cat{D}$, $\cat{D}'$ be triangulated categories, $w$, $w'$ weight structures on $\cat{C}$ and $\cat{C}'$, and $t$, $t'$ \textbf{t}-structures on $\cat{D}$ and $\cat{D}'$, respectively.
			Let $\Phi\: \cat{C}\times\cat{D}\to\cat{A}$ and $\Phi'\: \cat{C}'\times\cat{D}'\to\cat{A}$ be dualities and suppose that $w$ (resp.~$w'$) is \textbf{w}-\textbf{t}-strictly left orthogonal to $t$ (resp.~$t'$) with respect to $\Phi$ (resp.~$\Phi'$).
			Let $F\: \cat{C}\to\cat{C}'$ and $G\: \cat{D}'\to\cat{D}$ be ``$\Phi$-$\Phi'$-adjoint'' in the sense that $\Phi'(F(X),Y)\cong\Phi(X,G(Y))$ naturally for $X\in\cat{C}$ and $Y\in\cat{D}'$.
			Then
			\begin{enumerate}
				\item $F(\cat{C}_{w>0})\subseteq\cat{C}'_{w>0}$ if and only if $G(\cat{D}'^{t'\leq 0})\subseteq\cat{D}^{t\leq 0}$.
				\item $F(\cat{C}_{w<0})\subseteq\cat{C}'_{w<0}$ if and only if $G(\cat{D}'^{t'\geq 0})\subseteq\cat{D}^{t\geq 0}$.
			\end{enumerate}
			In particular, $F$ is weight exact if and only if $G$ is \textbf{t}-exact.
		\end{thm}
		\begin{proof}
			We only show the first part as the argument for the second claim is entirely analogous.
		
			Suppose $F(\cat{C}_{w>0})\subseteq\cat{C}'_{w'>0}$ and let $Y\in\cat{D}'^{t'\leq 0}$.
			By assumption we have $\cat{D}^{t\leq 0}=(\cat{C}_{w>0})^{\perp_\Phi}$, and thus
			\begin{equation*}
				G(Y)\in\cat{D}^{t\leq 0}\iff\cat{C}_{w>0}\perp_\Phi G(Y)\iff F(\cat{C}_{w>0})\perp_{\Phi'}Y.
			\end{equation*}
			But this condition is satisfied since by assumption $F(\cat{C}_{w>0})\subseteq\cat{C}'_{w'>0}$, and $Y\in\cat{D}'^{t'\leq 0}=(\cat{C}'_{w'>0})^{\perp_{\Phi'}}$.
			
			For the converse suppose $G(\cat{D}'^{t'\leq 0})\subseteq\cat{D}^{t\leq 0}$ and let $X\in\cat{C}_{w>0}$.
			By assumption we have $\cat{C}'_{w'>0}=\prescript{\perp_{\Phi'}}{}{(\cat{D}'^{t'\leq 0})}$, and thus
			\begin{equation*}
				F(X)\in\cat{C}'_{w'>0}\iff F(X)\perp_{\Phi'}\cat{D}'^{t'\leq 0}\iff X\perp_\Phi G(\cat{D}'^{t'\leq 0}).
			\end{equation*}
			But this is satisfied since by assumption $G(\cat{D}'^{t'\leq 0})\subseteq\cat{D}^{t\leq 0}$ and $X\in\cat{C}_{w>0}=\prescript{\perp_\Phi}{}{(\cat{D}^{t\leq 0})}$.
		\end{proof}
		Most notably it follows that the bijection between bounded \textbf{t}-structures with finite-length heart and bounded weight structures from \cite{KoenigYang} is a natural correspondence.
		To make this precise, following \cite{Chen} we write $\KbCAT$ for the strict $2$-category whose objects are the Hom-finite finite-length abelian categories with enough projectives and finitely many simples, with $1$-morphisms given by the functors $\Kb(\Proj(\cat{A}))\to\Kb(\Proj(\cat{B}))$ and $2$-morphisms the natural transformations between these.
		The $2$-category $\DbCAT$ is defined similarly, using functors $\Db(\cat{A})\to\Db(\cat{B})$ instead.
		Let $(\DbCAT)^\coop$ denote the bidual of $\DbCAT$, i.e.~the $2$-category obtained by reversing all $1$-morphisms and $2$-morphisms.
		By \cite[Thm.~3.2]{Chen} there is an equivalence $\KbCAT\to(\DbCAT)^\coop$ that is the identity on objects and sends a $1$-morphism (i.e.~a functor) $F\: \Kb(\Proj(\cat{B}))\to\Kb(\Proj(\cat{A}))$ to its \emph{right pseudo-adjoint} $F^\vee\: \Db(\cat{A})\to\Db(\cat{B})$, which is defined by natural isomorphisms
		\begin{equation*}
			\Hom_{\Db(\cat{A})}(F(X),Y)\cong\Hom_{\Db(\cat{B})}(X,F^\vee(Y))
		\end{equation*}
		for $X\in\Kb(\Proj(\cat{B}))$ and $Y\in\Db(\cat{A})$.
		\begin{cor}
			\label{naturalityfdalg}
			Let $\cat{A}$, $\cat{B}$ be Hom-finite finite-length abelian categories with enough projectives and finitely many simples.
			Let $t$, $t'$ be bounded \textbf{t}-structures on $\Db(\cat{A})$ and $\Db(\cat{B})$, respectively, and $w$, $w'$ the bounded weight structures on $\Kb(\Proj(\cat{A}))$ and $\Kb(\Proj(\cat{B}))$ corresponding to $t$ and $t'$ under the bijections from \cref{weighttbijections}.
			Assume that under the equivalence from \cite[Thm.~3.2]{Chen}, $G\: \Db(\cat{A})\to\Db(\cat{B})$ corresponds to $F\: \Kb(\Proj(\cat{B}))\to\Kb(\Proj(\cat{A}))$.
			Then $G$ is \textbf{t}-exact if and only if $F$ is weight exact.
		\end{cor}
		\begin{proof}
			By \cref{weighttbijections} the weight structure $w$ (resp.~$w'$) is \textbf{w}-\textbf{t}-strictly left orthogonal to the \textbf{t}-structure $t$ (resp.~$t'$).
			Thus the result follows from \cref{orthogonalitynaturality}, since the construction of $G$ as a right pseudo-adjoint to $F$ is precisely the required adjunction property.
		\end{proof}
		\begin{rem}
			Using \cite[Prop.~3.4]{Chen}, we also obtain the following consequence of \cref{naturalityfdalg}:
			\begin{enumerate}
				\item Suppose $F\: \Kb(\Proj(\cat{B}))\to \Kb(\Proj(\cat{A}))$ is weight exact and admits a right adjoint $G$.
					Then $G$ extends to $\tilde{G}\: \Db(\cat{A})\to \Db(\cat{B})$, and $\tilde{G}$ is \textbf{t}-exact.
				\item If $\tilde{G}\: \Db(\cat{A})\to \Db(\cat{B})$ is \textbf{t}-exact and restricts to $G\: \Kb(\Proj(\cat{A}))\to \Kb(\Proj(\cat{B}))$, then $G$ admits a left adjoint $F$, which is weight exact.
			\end{enumerate}
			In this setup, there is an alternative proof of 2):
			By \cref{orthogonalitylemma} we have $\cat{C}_{w<0}=\cat{D}^{t<0}\cap \Kb(\Proj(\cat{A}))$, and analogously for $t'$ and $w'$.
			Now let $X\in\cat{C}'_{w'\geq 0}$ and $Y\in\cat{C}_{w<0}=\cat{D}^{t<0}\cap\Kb(\Proj(\cat{A}))$.
			Then
			\begin{equation*}
				G(Y)=\tilde{G}(Y)\in\cat{D}'^{t'<0}\cap\Kb(\Proj(\cat{B}))=\cat{C}'_{w'<0}
			\end{equation*}
			by \textbf{t}-exactness of $G$, and from the adjunction and $\cat{C}'_{w'\geq 0}\perp\cat{C}'_{w'<0}$ we get
			\begin{equation*}
				\Hom_{\Kb(\Proj(\cat{A}))}(F(X),Y)=\Hom_{\Kb(\Proj(\cat{B}))}(X,G(Y))=0,
			\end{equation*}
			so $F(X)\in\prescript{\perp}{}{(\cat{C}_{w<0})}=\cat{C}_{w\geq 0}$.
			
			Similarly for $X\in\cat{C}'_{w'\leq 0}=\cat{D}'^{t'\leq 0}\cap\Kb(\Proj(\cat{B}))$ and $Y\in\cat{D}^{t>0}$ we have
			\begin{equation*}
				\Hom_{\Db(\cat{A})}(F(X),Y)=\Hom_{\Db(\cat{B})}(X,\tilde{G}(Y))=0,
			\end{equation*}
			and thus $F(X)\in\cat{C}_{w\leq 0}=\cat{D}^{t\leq 0}\cap \Kb(\Proj(\cat{A}))$.
		\end{rem}
	\printbibliography

@string{AMS="American Mathematical Society, Providence, RI"}

@article{BondarkoMotivicSpectralSequences,
	author={Bondarko, Mikhail Vladimirovich},
	title={Motivically functorial coniveau spectral sequences; direct summands of cohomology of function fields},
	journal={Documenta Math.},
	volume={Extra Vol.},
	year={2010}
}

@article{KoenigYang,
	author={Koenig, Steffen and Yang, Dong},
	title={Silting objects, simple-minded collections, $t$-structures and co-$t$-structures for finite-dimensional algebras},
	journal={Documenta Math.},
	volume={19},
	year={2014}
}

@incollection{BruestleYang,
	author={Br{\"u}stle, Thomas and Yang, Dong},
	title={Ordered exchange graphs},
	series={EMS Series of Congress Reports},
	booktitle={Advances in representation theory of algebras},
	editor={Benson, David John and Krause, Henning and Skowro{\'n}ski, Andrzej},
	publisher={European Mathematical Society},
	year={2014}
}

@misc{BruestleYangUpdate,
	author={Br{\"u}stle, Thomas and Yang, Dong},
	title={Ordered exchange graphs},
	eprinttype={arXiv},
	eprint={1302.6045v6},
	eprintclass={math.RT},
	year={2023}
}

@misc{BondarkoWeightTAndBack,
	author={Bondarko, Mikhail Vladimirovich},
	title={From weight structures to (orthogonal) t-structures and back},
	howpublished={Preprint},
	eprinttype={arXiv},
	eprint={1907.03686v1},
	eprintclass={math.KT},
	year={2019}
}

@article{CoelhoSimoesPauksztelloPloog,
	author={Coelho Sim{\~o}es, Raquel and Pauksztello, David and Ploog, David},
	title={Functorially finite hearts, simple-minded systems in negative cluster categories, and noncrossing partitions},
	journal={Compositio Math.},
	volume={158},
	year={2022}
}

@article{Chen,
	author={Chen, Xiao-Wu},
	title={Representability and autoequivalence groups},
	journal={Math. Proc. Camb. Phil. Soc.},
	volume={171},
	number={3},
	year={2021}
}

@article{BBD,
	author={Beilinson, Alexander and Bernstein, Joseph and Deligne, Pierre},
	title={Faisceaux pervers},
	journal={Ast{\'e}\-risque},
	volume={100},
	year={1982}
}

@article{Bondarko,
	author={Bondarko, Mikhail Vladimirovich},
	title={Weight structures vs.~$t$-structures; weight filtrations, spectral sequences, and complexes (for motives and in general)},
	journal={J. $K$-Theory},
	volume={6},
	year={2010}
}

@article{Pauksztello,
	author={Pauksztello, David},
	title={Compact corigid objects in triangulated categories and co-$t$-structures},
	journal={Centr. Eur. J. Math.},
	volume={6},
	number={1},
	year={2008}
}

@article{AlNofayee,
	author={Al-Nofayee, Salah},
	title={Simple objects in the heart of a $t$-structure},
	journal={J. Pure Appl. Algebra},
	volume={213},
	number={1},
	year={2009}
}

@article{RickardEquivalencesDerivedCategories,
	author={Rickard, Jeremy},
	title={Equivalences of derived categories for symmetric algebras},
	journal={J. Algebra},
	volume={257},
	number={2},
	year={2002}
}

@misc{Schnuerer,
	author={Schn{\"u}rer, Olaf M.},
	title={Simple-minded subcategories and $t$-structures},
	note={Unpublished note},
	url={https://math.uni-paderborn.de/fileadmin-eim/mathematik/AG-Algebra/Notes/simple-minded-subcategories-and-t-structures.pdf},
	year={2020}
}

@article{PsaroudakisVitoria,
	author={Psaroudakis, Chrysostomos and Vit{\'o}ria, Jorge},
	title={Realisation functors in tilting theory},
	journal={Math. Z.},
	volume={288},
	year={2018}
}

@misc{AHLSV,
	shorthand={AHLSV22},
	author={Angeleri H{\"u}gel, Lidia and Laking, Rosanna and {\v S}{\v t}ov{\'i}{\v c}ek, Jan and Vit{\'o}ria, Jorge},
	title={Mutation and torsion pairs},
	howpublished={Preprint},
	eprinttype={arXiv},
	eprint={2201.02147v1},
	eprintclass={math.RT},
	year={2022}
}

@misc{StacksProject,
	shorthand={Stacks},
	author={{The Stacks Project Authors}},
	title={\textit{Stacks Project}},
	howpublished={\url{https://stacks.math.columbia.edu}},
	year={2018},
}

@article{KellerVossieck,
	author={Keller, Bernhard and Vossieck, Dieter},
	title={Aisles in derived categories},
	journal={Bull. Soc. Math. Belg.},
	volume={40},
	year={1988}
}

@article{AiharaIyama,
	author={Aihara, Takuma and Iyama, Osamu},
	title={Silting mutation in triangulated categories},
	journal={J. Lond. Math. Soc. (2)},
	volume={85},
	number={3},
	year={2012}
}

@book{KrauseHomologicalTheory,
	author={Krause, Henning},
	title={Homological theory of representations},
	series={Cambridge Studies in Advanced Mathematics},
	number={195},
	publisher={Cambridge University Press},
	year={2021}
}

@book{Achar,
	author={Achar, Pramod N.},
	title={Perverse sheaves and applications to representation theory},
	series={Mathematical Surveys and Monographs},
	number={258},
	publisher=AMS,
	year={2021}
}

@article{SuYang,
	author={Su, Hao and Yang, Dong},
	title={From simple-minded collections to silting collections via Koszul duality},
	journal={Algebr. Represent. Theor.},
	volume={22},
	year={2019}
}

@article{Zhang,
	author={Zhang, Houjun},
	title={The ST correspondence for proper non-positive dg algebras},
	journal={Comm. Algebra},
	volume={51},
	number={11},
	year={2023}
}

@article{KellerDerivingDGCats,
	author={Keller, Bernhard},
	title={Deriving DG categories},
	journal={Ann. Sci. {\'E}c. Norm. Sup{\'e}r. (4)},
	volume={27},
	number={1},
	year={1994}
}

@article{KellerNicolas,
	author={Keller, Bernhard and Nicol{\'a}s, Pedro},
	title={Weight structures and simple dg modules for positive dg algebras},
	journal={Int. Math. Res. Not. IMRN},
	volume={2013},
	number={5},
	year={2013}
}

@incollection{KrauseChicago,
	author={Krause, Henning},
	title={Derived categories, resolutions, and Brown representability},
	booktitle={Interactions between homotopy theory and algebra},
	editor={Avramov, Luchezar L.  and Christensen, J. Daniel and Dwyer, William G. and Mandell, Michael A. and Shipley, Brooke E.},
	series={Contemporary Mathematics},
	number={436},
	year={2007}
}

@misc{KalckYangII,
	author={Kalck, Martin and Yang, Dong},
	title={Relative singularity categories II: dg models},
	howpublished={Preprint},
	eprinttype={arXiv},
	eprint={1803.08192v1},
	eprintclass={math.AG},
	year={2018}
}

@article{GenoveseRamos,
	author={Genovese, Francesco and Ramos Gonz{\'a}lez, Julia},
	title={A derived Gabriel-Popescu theorem for t-structures via derived injectives},
	journal={Int. Math. Res. Not. IMRN},
	volume={2023},
	number={6},
	year={2023}
}

@article{GenoveseLowenvdBergh,
	shorthand={GLVdB21},
	author={Genovese, Francesco and Lowen, Wendy and Van den Bergh, Michel},
	title={$t$-structures and twisted complexes on derived injectives},
	journal={Adv. Math.},
	volume={387},
	year={2021}
}

@article{AdachiMizunoYang,
	author={Adachi, Takahide and Mizuno, Yuya and Yang, Dong},
	title={Discreteness of silting objects and $t$-structures in triangulated categories},
	journal={Proc. London Math. Soc.},
	volume={118},
	number={3},
	year={2019}
}

@misc{Fushimi,
	author={Fushimi, Riku},
	title={The correspondence between silting objects and $t$-structures for non-positive dg algebras},
	howpublished={Preprint},
	eprinttype={arXiv},
	eprint={2312.17597v2},
	eprintclass={math.RT},
	year={2024}
}

@article{BeilinsonGinzburgSoergel,
	author={Beilinson, Alexander and Ginzburg, Victor and Soergel, Wolfgang},
	title={Koszul duality patterns in representation theory},
	journal={J. Amer. Math. Soc.},
	volume={9},
	number={2},
	year={1996}
}

@article{MazorchukOvsienkoStroppel,
	author={Mazorchuk, Volodymyr and Ovsienko, Serge and Stroppel, Catharina},
	title={Quadratic duals, Koszul dual functors, and applications},
	journal={Trans. Amer. Math. Soc.},
	volume={361},
	number={3},
	year={2009}
}

@article{EberhardtStroppel,
	author={Eberhardt, Jens Niklas and Stroppel, Catharina},
	title={Motivic Springer Theory},
	journal={Indag. Math.},
	volume={33},
	number={1},
	year={2022}
}

@misc{KellerNicolasCluster,
	author={Keller, Bernhard and Nicol{\'a}s, Pedro},
	title={Notes of the talk ``Cluster-hearts and cluster-tilting objects''},
	note={Unpublished note},
	url={https://pnp.mathematik.uni-stuttgart.de/iaz/iaz1/activities/t-workshop/NicolasNotes.pdf},
	year={2011}
}

@misc{FushimiDgKoszul,
	author={Fushimi, Riku},
	title={Contravariant Koszul duality between non-positive and positive dg algebras},
	howpublished={Preprint},
	eprinttype={arXiv},
	eprint={2409.08842v2},
	eprintclass={math.RT},
	year={2025}
}

@article{BondalKapranov,
    author={Bondal, Alexey I. and Kapranov, Mikhail M.},
    title={Enhanced triangulated categories},
    journal={Math. USSR, Sb.},
    volume={70},
    number={1},
    year={1991},
    //pages={93--107},
}

@misc{LurieHigherAlgebra,
    author={Lurie, Jacob},
    title={Higher Algebra},
    howpublished={Unpublished},
    url={https://people.math.harvard.edu/~lurie/papers/HA.pdf},
    //year={2017},
    date={2017-09-18},
}

@misc{LurieSpectralAlgebraicGeometry,
    author={Lurie, Jacob},
    title={Spectral Algebraic Geometry},
    howpublished={Unpublished},
    url={https://www.math.ias.edu/~lurie/papers/SAG-rootfile.pdf},
    //year={2018},
    date={2018-02-03},
}
\end{document}